\documentclass[12pt]{amsart}
\usepackage{amsmath,amssymb}
\usepackage{enumerate}
\usepackage{graphicx}
\usepackage{color}
\usepackage{auto-pst-pdf} 

\theoremstyle{plain}
\newtheorem{theorem}{Theorem}[section]
\newtheorem{lemma}[theorem]{Lemma}
\theoremstyle{definition}
\newtheorem{remark}[theorem]{Remark}
\newtheorem{problem}[theorem]{Problem}

\newcommand{\R}{\mathbb{R}}
\newcommand{\Z}{\mathbb{Z}}
\newcommand{\N}{\mathbb{N}}
\newcommand{\B}{\mathcal{B}}
\newcommand{\CPMM}{\mathcal{CPMM}}
\newcommand{\eps}{\varepsilon}

\begin{document}

\title{Constant slope maps on the extended real line}

\author[MICHA{\L} MISIUREWICZ]{MICHA{\L} MISIUREWICZ${}^\dagger$}
\author[SAMUEL ROTH]{SAMUEL ROTH$^\mathsection$}
\address{${}^\dagger{}^\mathsection$Department of Mathematical Sciences,
Indiana University-Purdue University Indianapolis,
402 N. Blackford Street, Indianapolis, IN 46202, USA}
\address{${}^\dagger{}^\mathsection$Institute of Mathematics, Polish Academy of Sciences,
\'Sniadeckich 8, 00-656 Warsaw, Poland}
\address{${}^\mathsection$Mathematical Institute, Silesian University,
Na Rybn\'{i}\v{c}ku 1, Opava 746 01, Czech Republic}
\email{${}^\dagger$mmisiure@math.iupui.edu}
\email{${}^\mathsection$samuel.roth@math.slu.cz}


\begin{abstract}
For a transitive countably piecewise monotone Markov interval map we
consider the question whether there exists a conjugate map of constant
slope. The answer varies depending on whether the map is continuous or
only piecewise continuous, whether it is mixing or not, what slope we
consider, and whether the conjugate map is defined on a bounded
interval, half-line or the whole real line (with the infinities included).
\end{abstract}


\maketitle

\section{Introduction}\label{sec:intro}

Among piecewise monotone interval maps, the simplest to understand are
the piecewise linear maps with the same absolute value of slope on
every piece; such maps are said to have \emph{constant slope} (and we
will usually say ``slope'' in the meaning ``absolute value of slope''). These
maps offer the dynamicist many advantages. For example, if we wish to
compute the topological entropy, we just take the larger of zero and
the logarithm of the slope. If we wish to study the symbolic dynamics
and the slope is larger than one, then no two points can share the
same itinerary; this already rules out wandering intervals.

There are two classic results by which constant slope maps provide a
good model for understanding more general piecewise monotone maps.
Nearly fifty years ago, Parry proved that every topologically
transitive, piecewise monotone (and even piecewise continuous)
interval map is conjugate to a map of constant slope~\cite{Pa}. Then,
in the 1980's, Milnor and Thurston showed how to modify Parry's
theorem to remove the hypothesis of topological transitivity
\cite{MT}. As long as a piecewise monotone map has positive
topological entropy $\log\lambda$, there exists a semiconjugacy to a map
of constant slope $\lambda$. The semiconjugating map is
nondecreasing, preserving the order of points in the interval, but
perhaps collapsing some subintervals down to single points.

It is natural to ask how well the theory extends to \emph{countably
  piecewise monotone} maps, when we no longer require the set of
turning points to be finite, but still require it to have a countable
closure. The theory has to be modified in several ways. In contrast to
Parry's result, it is possible to construct transitive examples which
are not conjugate to any interval map of constant slope. Such examples
are contained in the authors' prior publication~\cite{MR}. However,
those particular examples are conjugate to constant slope maps on the
extended real line $[-\infty,\infty]$, which we may choose to regard
as an interval of infinite length. In response, Bobok and Bruin posed
the following problem: \textit{Under what conditions does a countably
  piecewise monotone interval map admit a conjugacy to a map of
  constant slope on the extended real line?}

The present paper answers this problem, focusing on the Markov case.
Section~\ref{sec:def} gives the necessary definitions, then closes
with a theorem due to Bobok~\cite{Bo}, that there exists a
nondecreasing semiconjugacy to a map of constant slope $\lambda > 1$
on the \emph{finite length} interval $[0,1]$ if and only if the
relevant transition matrix admits a nonnegative eigenvector with
eigenvalue $\lambda$ and summable entries.

We investigate whether eliminating the summability requirement on the
eigenvector might allow for the construction of constant slope models
on the extended real line. We find two new obstructions, not present
in previous works. We present one example which is topologically
transitive but not mixing (Section~\ref{sec:mix}), and one example
which is mixing but only piecewise continuous
(Section~\ref{sec:pcws}), and we show for both of these examples that
although the transition matrix admits a nonnegative eigenvector with
some eigenvalue $\lambda$, nevertheless there is no conjugacy to any
map of constant slope $\lambda$. These are the only two obstructions;
under the assumptions of continuity and topological mixing, we state
in Section~\ref{sec:main} our main theorem, Theorem~\ref{thm:main},
that a countably piecewise monotone and Markov map which is continuous
and mixing admits a conjugacy to a map of constant slope $\lambda$ on
some non-empty, compact (sub)interval of the extended real line
$[-\infty,\infty]$ if and only if the associated Markov transition
matrix admits a nonnegative eigenvector of eigenvalue $\lambda$.
Sections~\ref{sec:begin} through~\ref{sec:finish} are dedicated to the
proof of Theorem~\ref{thm:main}.

Sections~\ref{sec:example} through~\ref{sec:example2} apply
Theorem~\ref{thm:main} to three explicit examples of continuous
countably piecewise monotone Markov maps. The first
admits conjugate maps of constant slope on the unit interval.  The second
admits conjugate maps of constant slope on the extended real line and the
extended half line.  The third does not admit any conjugate map of constant
slope.  These sections illustrate a variety of novel techniques for calculating
the nonnegative eigenvectors of a countable 0-1 matrix and for calculating
the topological entropy of a countably piecewise monotone map.

\section{Definitions and Background}\label{sec:def}

The extended real line $[-\infty,\infty]$ is the ordered set
$\R\cup\{\infty,-\infty\}$ equipped with the order topology; this
topological space is a two-point compactification of the real line and
is homeomorphic to the closed unit interval $[0,1]$.

There is an exceedingly simple example which illustrates how the
extended real line behaves differently than the unit interval with
respect to constant slope maps. Consider the map $f:[0,1]\to[0,1]$
given by $f(x)=x^2$. A conjugate map on the unit interval must be
monotone with fixed points at $0$ and $1$, and the only constant slope
map with those properties is the identity map. On the other hand, the
map $g:[-\infty,\infty]\to[-\infty,\infty]$ given by $g(x)=x-1$ has
constant slope $1$ and is conjugate to $f$ by the homeomorphism
$h:[0,1]\to[-\infty,\infty]$, $h(x)=-\log_2(-\log_2(x))$. Thus, we
achieve constant slope by taking advantage of the infinite length of
the extended real line and pushing our fixed points out to $\pm
\infty$, see Figure~\ref{fig:ext}.

\begin{figure}[ht]
\begin{center}
\hfill
\includegraphics[height=2in]{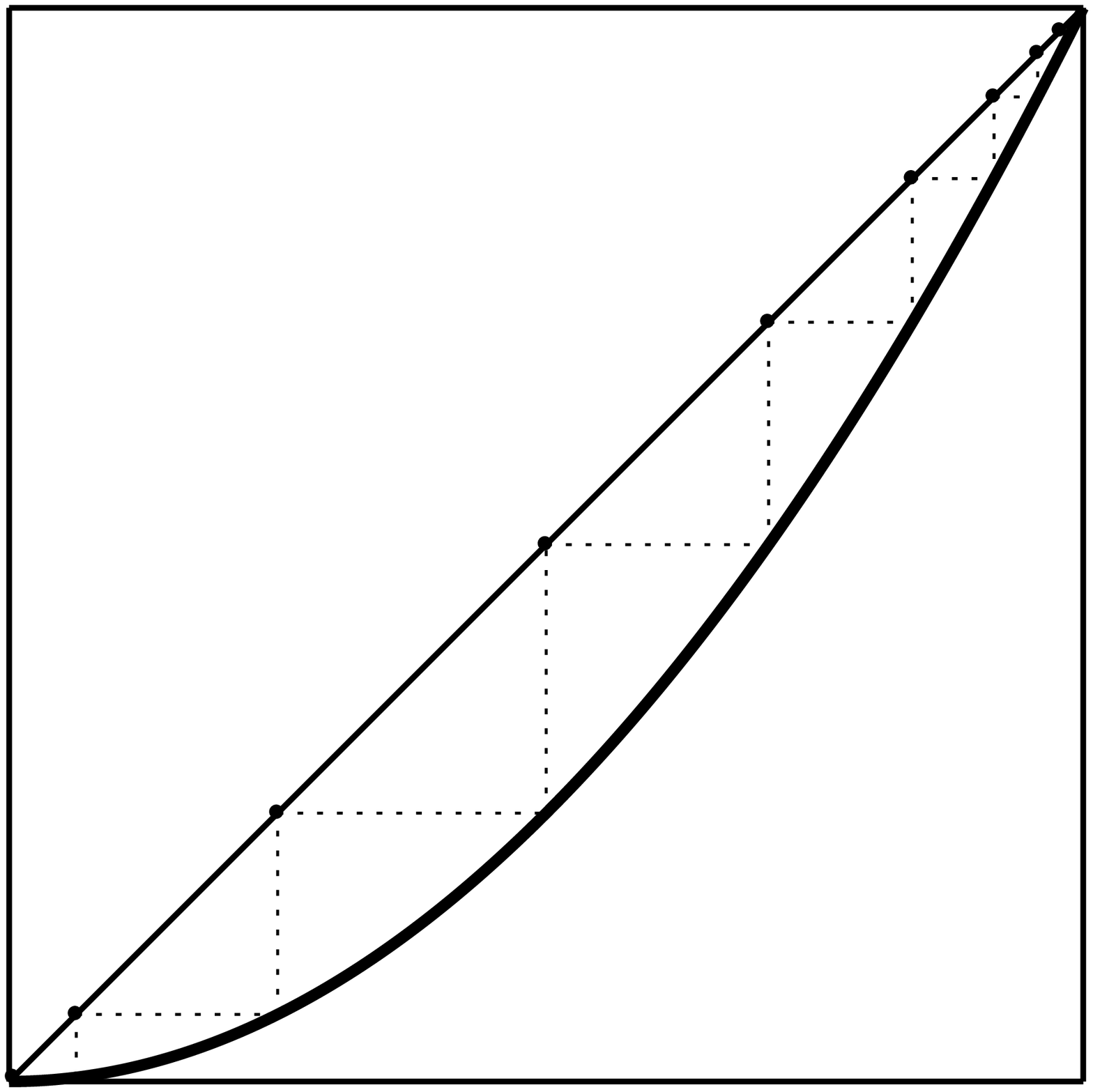}
\hfill
\includegraphics[height=2in]{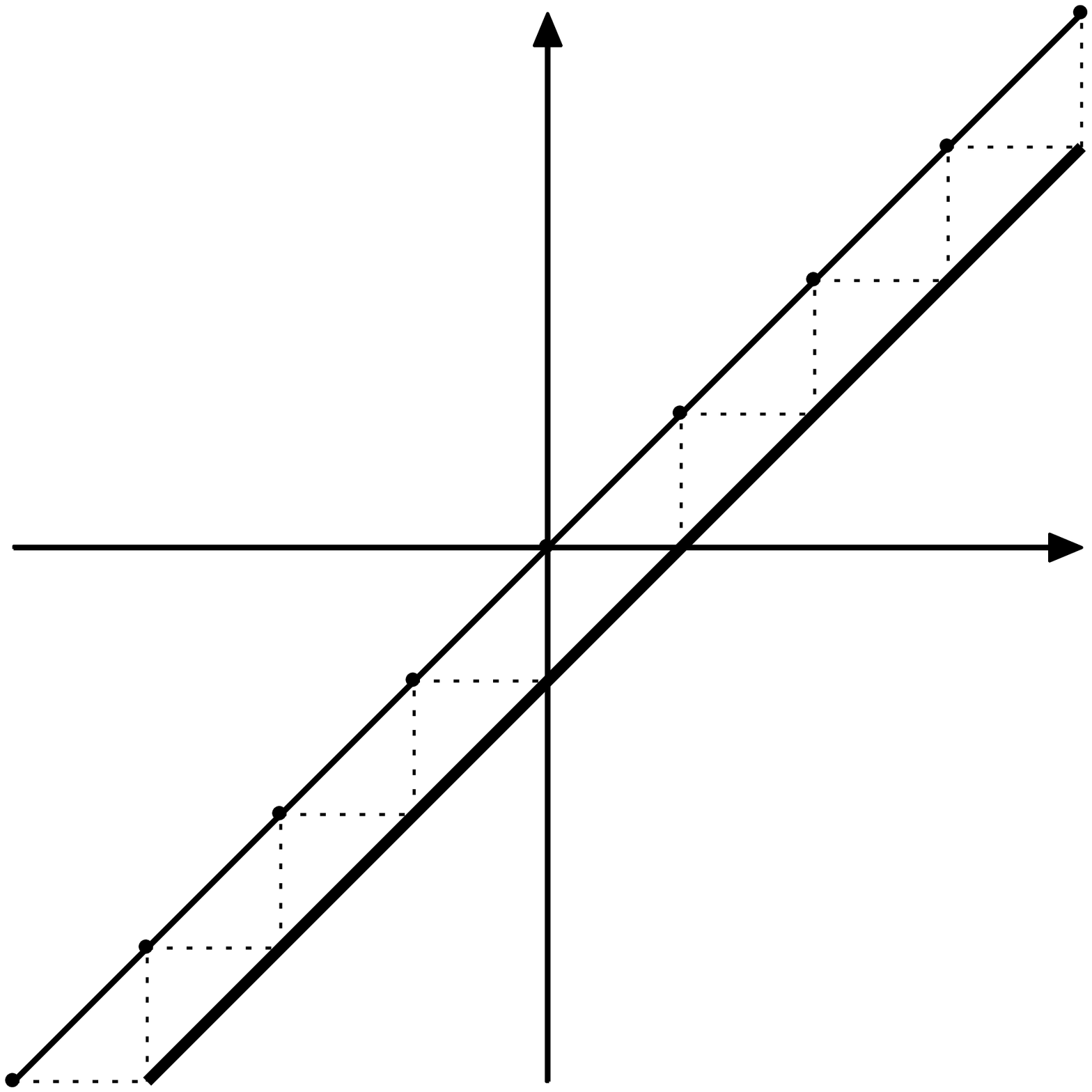}
\hfill
\phantom{.}
\caption{A conjugacy which moves the fixed points to
  $\pm\infty$.}\label{fig:ext}
\end{center}
\end{figure}

Suppose $f$ is a continuous self-map of some interval $[a,b]$,
$-\infty\leq a < b \leq \infty$, and suppose there exists a closed,
countable set $P\subset[a,b]$, $a,b\in P$, such that $f(P)\subset P$
and $f$ is monotone on each component of $[a,b]\setminus P$. Such a
map is said to be \emph{countably piecewise monotone and Markov} with
respect to the \emph{partition set} $P$; the components of $[a,b]\setminus P$ are called
$P$-\emph{basic intervals}, and the set of all $P$-basic intervals is
denoted $\B(P)$. If additionally the restriction of $f$ to each
$P$-basic interval is affine with slope of absolute value $\lambda$,
then we say that $f$ has \emph{constant slope} $\lambda$. This is a
geometric, rather than a topological property, and it is the reason we
must distinguish finite from infinite length intervals. The class of
all continuous countably piecewise monotone and Markov maps is denoted $\CPMM$.
The subclass of those maps which act on the closed unit interval
$[0,1]$ is denoted $\CPMM_{[0,1]}$.

Let us draw the reader's attention to three properties of the class
$\CPMM$. First, the underlying interval $[a,b]$ depends on the map $f$
and is permitted to be infinite in length. Second, the map $f$ is
required to be globally (rather than piecewise) continuous; this is
essential for our use of the intermediate value theorem. And third,
the set $P$ is required to be forward-invariant. This is the Markov
condition; it means that if $I$, $J$ are $P$-basic intervals and
$f(I)\cap J\neq\emptyset$, then $f(I)\supseteq J$.

If $f$ is countably piecewise monotone and Markov with respect to $P$,
then we define the binary \emph{transition matrix} $T=T(f,P)$ with
rows and columns indexed by $\B(P)$ and entries
\begin{equation}\label{transitionmatrix}
T(I,J)=\begin{cases}1,&\textnormal{ if }f(I)\supseteq
J\\0,&\textnormal{ otherwise.}\end{cases}
\end{equation}
This transition matrix represents a linear operator on the linear
space $\R^{\B(P)}$ without any reference to topology. In particular,
in Theorem~\ref{thm:Bobok} below, $T$ is not required to represent
a \emph{bounded} operator on $\ell^1(\R^{\B(P)})$, nor even to
preserve the subspace $\ell^1(\R^{\B(P)})\subset \R^{\B(P)}$.

We wish also to study maps which are only piecewise continuous. To
that end, we define the class $\CPMM^{pc}$ which contains any self-map
$f$ of any closed, nonempty interval $[a,b]\subset[-\infty,\infty]$,
such that $f$ has the following properties with respect to some
closed, countable set $P\subset[a,b]$, $a,b\in P$, namely, that $f(P)\subset P$, that the
restriction $f|_I$ is strictly monotone \emph{and continuous} for each
$P$-basic interval $I\in\B(P)$, and that for each pair $I,J\in\B(P)$,
either $f(I)\cap J=\emptyset$ or else $f(I)\supseteq J$. Note that
this definition places no continuity requirements on the map $f$ at
the points of $P$.\footnote{Although this basically means that we do
  not care what the values of $f$ at the points of $P$ are, we still
  need the assumption $f(P)\subset P$ in order for $P$ to function
  as a ``black hole.''} Studies on finitely piecewise continuous maps often
require one-sided continuity at the endpoints of the basic intervals
(allowing the map to take 2 values at a single point), but our
countable set $P$ may have accumulation points which are not the
endpoint of any $P$-basic interval, so such an approach makes no sense
here.
We define also the subclass $\CPMM^{pc}_{[0,1]}$ for those maps
which act on the closed unit interval $[0,1]$, and we define the
binary transition matrix $T=T(f,P)$ by the same
formula~\eqref{transitionmatrix} as before.

We will also have cause to consider the properties of topological
transitivity and topological mixing. We use
the standard definitions, that a map $f$ is \emph{topologically
  transitive} (respectively, \emph{topologically mixing}) provided
that for every pair of nonempty open sets $U, V$ there exists
$n_0\in\mathbb{N}$ such that $f^{n_0}(U)\cap V\neq\emptyset$
(respectively, for all $n\geq n_0$, $f^n(U)\cap V\neq\emptyset$).

If we ignore the extended real line $[-\infty,\infty]$ and allow only
for maps on the unit interval $[0,1]$, then there is already an
established necessary and sufficient condition to determine when a map
is semiconjugate to a map of constant slope.

\begin{theorem}\label{thm:Bobok}
\emph{(Bobok,~\cite{Bo})} Let $f\in\CPMM^{pc}_{[0,1]}$ be a map with
partition set $P$ and transition matrix
$T$, and fix $\lambda>1$. Then $f$ is semiconjugate via a continuous
nondecreasing map $\psi$ to some map $g\in\CPMM^{pc}_{[0,1]}$ of constant
slope $\lambda$ if and only if $T$ has a nonnegative eigenvector
$v=\left(v_I\right)\in\ell^1(\R^{\B(P)})$ with eigenvalue $\lambda$.
\end{theorem}

\begin{remark}
The statement of Theorem~\ref{thm:Bobok} in~\cite{Bo} is for the class
$\CPMM_{[0,1]}$ of globally continuous maps, but the proof applies
equally well in the piecewise continuous setting.
\end{remark}

We draw the reader's attention to the requirement
$v\in\ell^1(\R^{\B(P)})$ that the eigenvector should be summable. If
we read the proof in~\cite{Bo}, the reason for this is clear. If we
are given the semiconjugacy $\psi$ to the constant slope map, then we
construct the eigenvector $v$ by setting $v_I=|\psi(I)|$ for each
$P$-basic interval $I$, where $|\cdot|$ denotes the length of an
interval, and therefore the sum of the entries $v_I$ is just the
length of the unit interval $[0,1]$. Conversely, if we are given an
eigenvector $v$, then we rescale it so that the sum of entries is $1$
and then construct the semiconjugacy in such a way that
$|\psi(I)|=v_I$ for all $I$, obtaining a map $g$ of an interval of
length $1$.

\section{Statement of Main Results}\label{sec:main}

We return now to the question, when does a map $f\in\CPMM$ admit a
nondecreasing semiconjugacy $\psi$ to a map $g$ of constant
slope on some compact subinterval of $[-\infty,\infty]$, whether
finite or infinite in length? It is clear that $g$ must belong to the
class $\CPMM$, because $g$ will necessarily be piecewise monotone and
Markov with respect to $\psi(P)$ -- see~\cite[Lemma 4.6.1]{ALM}. Here
are the statements of our main results:

\begin{theorem}\label{thm:main}
Let $f\in\CPMM$ be a map with partition set $P$ and transition matrix
$T$, and fix $\lambda > 1$.
Assume that $f$ is topologically mixing. Then $f$ is conjugate via a
homeomorphism $\psi$ to some map $g\in\CPMM$ of constant slope
$\lambda$ if and only if
\begin{equation}\label{criterion}
T \textnormal{\textit{ has a nonnegative eigenvector }}
v=(v_I)\in\R^{\B(P)} \textnormal{\textit{ with eigenvalue }} \lambda.
\end{equation}
\end{theorem}

\begin{theorem}\label{thm:pcws}
In the piecewise continuous case (replacing $\CPMM$ by $\CPMM^{pc}$),
condition~\eqref{criterion} is necessary but not sufficient.
\end{theorem}

\begin{theorem}\label{thm:notmixing}
If we replace the hypothesis of topological mixing with the weaker
condition of topological transitivity, then condition
\eqref{criterion} is necessary but not sufficient.
\end{theorem}

\begin{remark}\label{rem:main}
Since the map $f$ in Theorem~\ref{thm:main} defines a topological
dynamical system without regard to geometry, there is no loss of
generality if we assume that $f\in\CPMM_{[0,1]}$. We will make this
assumption from now on.
\end{remark}

We will start by showing the necessity of condition~\eqref{criterion}.
The same proof applies in all cases (continuous or piecewise
continuous, mixing or transitive). Showing the sufficiency of
condition~\eqref{criterion} in Theorem~\ref{thm:main} requires much
more work. We give an explicit construction of the conjugating map
$\psi$ in several stages. Our construction begins very much like the
proofs of~\cite[Theorem 2.5]{Bo} and~\cite[Theorem 4.6.8]{ALM}, but
the unsummability of $v$ introduces some additional difficulties not
present in these previous works. It is the strength of global
continuity and topological mixing which allows us to overcome these
difficulties. The insufficiency of condition~\eqref{criterion} in
Theorems~\ref{thm:pcws} and~\ref{thm:notmixing} is proved by example
in Sections~\ref{sec:pcws} and~\ref{sec:mix}.

\section{The Proof Begins}\label{sec:begin}

\begin{lemma}\label{lem:Bobok}
\emph{(Bobok)} In Theorems~\ref{thm:main},~\ref{thm:pcws}, and
\ref{thm:notmixing}, condition~\eqref{criterion} is necessary.
\end{lemma}

\begin{proof}
This proof is due to private communication with Jozef Bobok. As
mentioned before, we may suppose that $f$ is defined on the finite
interval $[0,1]$. Let $\psi$ be the conjugating map, $\psi\circ
f=g\circ\psi$. Define $v$ by $v_I=|\psi(I)|$, $I\in\B(P)$, where
$|\cdot|$ denotes the length of an interval. A priori, we may have
$|\psi(I)|=\infty$; this happens if and only if $I$ contains one of
the endpoints $0,1$ and $\psi$ maps this endpoint to one of
$\pm\infty$. (Recall that if $0,1$ are accumulation points of $P$,
then they are not endpoints of any $P$-basic interval). We want to
show that all the entries of $v$ are finite. Since $g$ is monotone
with slope of absolute value $\lambda$ on each $\psi(P)$-basic
interval, we have
\begin{equation}\label{stretchLambda}
|g(\psi(I))|=\lambda|\psi(I)|, \qquad I\in\B(P),
\end{equation}
where if one side of the equality is infinite then so is the other.
Let $\mathcal{F}$ denote the collection of all $P$-basic intervals $I$
such that $|\psi(I)|=\infty$. If $I\in\mathcal{F}$ and if
$f(J)\supseteq I$, then by the conjugacy of $f$, $g$ and by
equation~\eqref{stretchLambda}, it follows that $J\in\mathcal{F}$. Now
invoke topological transitivity and the Markov condition, and it
follows that either $\mathcal{F}=\emptyset$ or $\mathcal{F}=\B(P)$.

Suppose toward contradiction that $\mathcal{F}=\B(P)$. Then there are
neighborhoods of $\pm\infty$ on which $g$ is affine with constant
slope $\lambda>1$. It follows that at least one of the points
$\pm\infty$ is an attracting fixed point, or else they form an
attracting two-cycle (slope larger than 1 in a neighborhood of
infinity means that the images of points close to infinity are even
closer to infinity). This contradicts transitivity. We may conclude
that $\mathcal{F}=\emptyset$ and all entries of $v$ are finite.

We still need to show that $v$ is an eigenvector for $T$. Applying
equation~\eqref{stretchLambda} we have
\begin{equation*}
\lambda v_I = \lambda |\psi(I)| =
|g(\psi(I))|=|\psi(f(I))|=\sum_{J\subset
  f(I)}|\psi(J)|=\sum_{J\in\B(P)}T_{IJ}v_J.
\end{equation*}

\end{proof}

Now we begin the long work of proving the sufficiency of condition
\eqref{criterion} in Theorem~\ref{thm:main}. Let $f$, $T$, be as in
the statement of the theorem, fix $\lambda>1$, and suppose
$Tv=\lambda v$ for some nonzero vector
$v=\left(v_I\right)\in\R^{\B(P)}$ with nonnegative entries. We will
assume (by Remark~\ref{rem:main}) that $f\in\CPMM_{[0,1]}$. We will
construct a map $\psi:[0,1]\to[-\infty,\infty]$ which is a
homeomorphism onto its image in such a way that $g:=\psi\circ
f\circ\psi^{-1}$ has constant slope $\lambda$. Define the sets
\begin{equation*}
P_n=\bigcup_{i=0}^n f^{-i}(P), \, n\in\N, \quad Q=\bigcup_{i=0}^\infty f^{-i}(P)
\end{equation*}

The set $Q$ is backward invariant by construction and forward
invariant because $P$ is forward invariant. $Q$ is a dense subset of
$[0,1]$ because $f$ is mixing. Choose a basepoint $p_0\in P$ and
define $\psi$ on $Q$ by the formula
\begin{equation}\label{PsiOnQ}
\def\arraystretch{1.5}
\psi(x)=
\left\{
\begin{array}{c@{\hskip1em}l}
0, & \textnormal{ if }x=p_0 \\ \lambda^{-n}
\sum\limits_{\substack{J\in B(P^n)\\p_0<J<x}}v_{f^n(J)},
& \textnormal{ if } x\in P_n, x>p_0 \\ -\lambda^{-n}
\sum\limits_{\substack{J\in B(P^n)\\x<J<p_0}}v_{f^n(J)},
& \textnormal{ if } x\in P_n, x<p_0
\end{array}
\right.
\end{equation}

The choice of $p_0$ is somewhat arbitrary, but to simplify the proof
of Lemma~\ref{lem:properties}~\eqref{finite}, we insist that $0<p_0<1$
and that $p_0$ is an endpoint of some $P$-basic interval (i.e., $p_0$
is not a 2-sided accumulation point of $P$). This is possible because
$P$ is a closed, countable subset of $[0,1]$ and hence cannot be
perfect.

\begin{remark}\label{rem:gPhaseSpace}
In light of equation~\eqref{PsiOnQ}, we find that we are constructing 
a map $g$ on
\begin{itemize}
\item a finite interval $[a,b]$, if $\sum v_I<\infty$,
\item an extended half-line $[a,\infty]$, if $\sum_{I<p_0}<\infty$ and $\sum_{I>p_0}=\infty$,
\item an extended half-line $[-\infty,b]$, if $\sum_{l<p_0}=\infty$ and $\sum_{l>p_0}<\infty$, and
\item the extended real line $[-\infty,\infty]$, if $\sum_{l<p_0}=\infty$ and $\sum_{l>p_0}=\infty$.
\end{itemize}
\end{remark}

\begin{lemma}\label{lem:properties}
The function $\psi:Q\to[-\infty,\infty]$ has the following properties:
\begin{enumerate}[(i)]
\item\label{welldef} $\psi$ is well-defined; i.e. when $x\in P_{n_1}$
  and $x\in P_{n_2}$, the sums agree.
\item\label{mono} $\psi$ is strictly monotone increasing.
\item\label{expansion} If $x, x'\in Q$ belong to an interval of
  monotonicity of $f$, then
\begin{equation*}
|\psi(f(x))-\psi(f(x'))|=\lambda|\psi(x)-\psi(x')|,
\end{equation*}
where if one side of the equality is infinite, then so is the other.
\item\label{subexpansion} For arbitrary $x,x'\in Q$ we have
\begin{equation*}
|\psi(f(x))-\psi(f(x'))|\leq\lambda|\psi(x)-\psi(x')|,
\end{equation*}
and we allow for the possibility that one or both sides of this
inequality are infinite.
\item\label{finite} For $0<x<1$, $\psi(x)$ is finite.
\end{enumerate}
\end{lemma}

\begin{proof}
\begin{enumerate}[(i)]
\item Suppose $K\in\B(P_n)$. Then $f^n|K$ is monotone and
  $f^n(K)\in\B(P)$. Therefore
\begin{multline*}
\qquad
\lambda^{-n-1}\sum\limits_{\substack{J\in\B(P_{n+1})\\J\subset
    K}}v_{f^{n+1}(J)} =
\lambda^{-n-1}\sum\limits_{\substack{J\in\B(P_1)\\J\subset
    f^n(K)}}v_{f(J)} = \\ = \lambda^{-n-1}\sum\limits_{J\in\B(P_0)}
T_{f^n(K)J} v_J = \lambda^{-n-1} \lambda v_{f^n(K)} = \lambda^{-n}
v_{f^n(K)}
\end{multline*}
This shows that $\psi$ is well-defined.
\item We will use the nonnegativity of the eigenvector $v$ together
  with the mixing hypothesis to show that the entries of $v$ must be
  strictly positive. Strict monotonicity of $\psi$ then follows from
  the definition. Since $v$ is not the zero vector, there must be some
  $P$-basic interval $I_0$ with $v_{I_0}\neq0$. Let $I\in\B(P)$. By
  the mixing hypothesis, there is $n\in\N$ such that
  $(T^n)_{II_0}\neq0$. Then $v_I=\lambda^{-n}\sum_J (T^n)_{IJ} v_J
  \geq \lambda^{-n} v_{I_0}>0$.
\item For $x,x'\in Q$ there exists a common value $n\geq1$
 such that $x,x'\in P_n$ (since $P_0\subset P_1\subset P_2
 \subset \cdots$).  Then $f(x),f(x')\in P_{n-1}$. By the monotonicity of 
 $f$ between $x,x'$, the assignment $K=f(J)$ defines a bijective correspondence
 \begin{equation*}
 \qquad \left\{J\in\B(P_n):J\mbox{ between }x,x'\right\} \quad \longleftrightarrow
 \quad \left\{K\in\B(P_{n-1}):K\mbox{ between }f(x),f(x')\right\}.
 \end{equation*}
 By the definition of $\psi$ we may sum over those sets and obtain
 \begin{multline*}
 \qquad |\psi(f(x))-\psi(f(x'))|=\sum_{K}\lambda^{-(n-1)} v_{f^{n-1}(K)}= \\
 = \sum_{J}\lambda^{-(n-1)}v_{f^{n-1}(f(J))}=\lambda|\psi(x)-\psi(x')|.
 \end{multline*}
\item This is the inequality that survives from~\eqref{expansion} when
  we allow for folding between $x$ and $x'$.  To see it, we imitate the proof
  of~\eqref{expansion}, noticing that by the intermediate value theorem the
  assignment $K=f(J)$ attains every interval $K$ between $f(x),f(x')$ at least once.
\item Let $x$ be given, $0<x<1$. Assume $x<p_0$; the proof when
  $x>p_0$ is similar. Fix a $P$-basic interval $J_0$ with $p_0$ at one
  endpoint. Because $f$ is mixing, there exists $n$ such that $J_0\cap
  f^{-n}((p_0,1))\neq\emptyset$ and $J_0\cap
  f^{-n}((0,x))\neq\emptyset$. By the intermediate value theorem there
  exist $x_1,x_2\in J_0$ with $f^n(x_1)=x$ and $f^n(x_2)=p_0$. By
 ~\eqref{subexpansion} applied $n$ times, $|\psi(x)|\leq
  \lambda^n|\psi(x_2)-\psi(x_1)|$. But by~\eqref{mono},
  $|\psi(x_2)-\psi(x_1)|\leq|\psi(\sup J_0)-\psi(\inf J_0)|$. At the
  two endpoints of $J_0$, $\psi$ takes the finite values $0$ and
  $v_{J_0}$ (or possibly $-v_{J_0}$).
\end{enumerate}
\end{proof}

The main problem to tackle before we can extend $\psi$ to the desired
homeomorphism is to show that the map we have defined so far has no
jump discontinuities.

\begin{problem}\label{prob:jump}
Show that for each $x\in [0,1]$,
\begin{equation*}
\inf \psi(Q\cap(x,1]) = \sup \psi(Q\cap[0,x)),
\end{equation*}
except that for $x=0$ we write $\psi(0)$ in place of the supremum and
for $x=1$ we write $\psi(1)$ in place of the infimum.
\end{problem}

The resolution of this problem makes essential use of the global
continuity of $f$ as well as the order structure of the interval
$[0,1]$. Moreover, special treatment is required for the points $x\in
Q$ -- we must show the continuity of $\psi$ from each side separately.
We do this by introducing a notion of ``half-points.''
\footnote{It is a slight modification of the construction
  from~\cite{M-ihes}. However, really the idea goes back to the
  International Mathematical Olympiad in 1965, where the Polish team
  was making jokes about the half-points $[a,a)$ and $(a,a]$.}

\section{Half-Points}\label{sec:half}

Construct the sets
\begin{equation*}
\tilde{Q}=\left(Q \times \{+,-\}\right)\setminus\{(0,-),(1,+)\}, \quad
S= \left([0,1] \setminus Q\right) \cup \tilde{Q}
\end{equation*}

The way to think of this definition is that we are splitting each
point $x\in Q$ into the two \emph{half-points} $(x,+)$ and $(x,-)$.
$S$ is the interval $[0,1]$ with each point of $Q$ replaced by
half-points. We use boldface notation to represent points in $S$,
whether half or whole. Thus, $\boldsymbol{x}$ may mean $x$ or $(x,+)$
or $(x,-)$, depending on the context.

Let us extend the dynamics of $f$ from $[0,1]$ to $S$. Recall that $Q$
is both forward and backward invariant. On $S\setminus \tilde{Q} =
[0,1]\setminus Q$ we keep the map $f$ without change. To extend $f$
from $Q$ to $\tilde{Q}$ we define a notion of the orientation of the
map at half-points. We say that $f$ is \emph{orientation-preserving}
(resp. \emph{orientation-reversing}) at the half-point $(x,+)$ if some
half-neighborhood $[x,x+\eps)$ is contained in some $J\in\B(P)$ with
  $f|_J$ increasing (resp. decreasing). For a half-point $(x,-)$, the
  definition is the same, except that we look at a half-neighborhood
  of the form $(x-\eps,x]$. It is not clear how to decide if $f$ is
orientation-preserving or orientation-reversing at the accumulation
points of $P$. It may happen that every half-neighborhood of $x$
contains $f(x)$ in the interior of its image, so that neither
definition is appropriate. Nevertheless, we define the extended map
$f$ on $\tilde{Q}$ by the following formula:
\begin{equation}\label{extendf}
\begin{aligned}
f(x,+)&=
\begin{cases}
(f(x),+),&\textnormal{ if } f \textnormal{ is orientation-preserving
    at } (x,+) \\ (f(x),-),&\textnormal{ if } f \textnormal{ is
    orientation-reversing at } (x,+) \\ (f(x),+),&\textnormal{ if }
  \forall_{\eps>0}\;\exists_{x'\in P\cap[x,x+\eps)}\;f(x')>f(x)
    \\ (f(x),-),&\textnormal{ otherwise}
\end{cases}\\
f(x,-)&=
\begin{cases}
(f(x),+),&\textnormal{ if } f \textnormal{ is orientation-reversing at
  } (x,-) \\ (f(x),-),&\textnormal{ if } f \textnormal{ is
    orientation-preserving at } (x,-) \\ (f(x),+),&\textnormal{ if }
  \forall_{\eps>0}\; \exists_{x'\in P\cap(x-\eps,x]}\; f(x')>f(x)
  \\ (f(x),-),&\textnormal{ otherwise}
\end{cases}\\
\end{aligned}
\end{equation}

Let us say a few words about the ``otherwise'' cases. Consider a
half-point $(x,+)$ which does not fit into any of the first three
cases. We claim that for such a point, $\forall_{\eps>0}\; \exists_
{x'\in P\cap[x,x+\eps)}\; f(x')<f(x)$. If not, we would have to
conclude that $\exists_{\eps>0} \; \forall_{x'\in P\cap[x,x+\eps)}\;
f(x')=f(x)$. But this is impossible, because the half-neighborhood
$[x,x+\eps)$ must contain some $P$-basic interval $J$, and by 
the strict monotonicity of $f|_J$ the two endpoints of this interval 
have distinct images. Similarly, if a half-point $(x,-)$ falls
into the ``otherwise'' case, then $\forall_{\eps>0}\;
\exists_{x'\in P\cap(x-\eps,x]}\; f(x')<f(x)$. This is relevant
in the proofs of Lemmas~\ref{lem:atomgrowth}
and~\ref{lem:intermediate}.

Now we define a real-valued function $\Delta_\psi$ on $S$ by the formula
\begin{equation*}
\Delta_\psi(\boldsymbol{x}) =
\begin{cases}
\inf \psi(Q\cap(x,1]) - \psi(x) & \textnormal{ if }
\boldsymbol{x}=(x,+) \in \tilde{Q}\\ \psi(x) - \sup \psi(Q\cap[0,x)) &
  \textnormal{ if } \boldsymbol{x}=(x,-) \in \tilde{Q}\\ \inf
  \psi(Q\cap(x,1]) - \sup \psi(Q\cap[0,x)) & \textnormal{ if }
  \boldsymbol{x} = x \in S \setminus \tilde{Q}
\end{cases}
\end{equation*}

If $\Delta_\psi(\boldsymbol{x})>0$, then we say that $\boldsymbol{x}$
is an \emph{atom} for $\psi$ and $\Delta_\psi(\boldsymbol{x})$ is its
\emph{mass}. In this language, Problem~\ref{prob:jump} asks us to show
that $\psi$ has no atoms.

The next lemma is an analog of
Lemma~\ref{lem:properties}~\eqref{expansion} for a single point (or
half-point) $\boldsymbol{x}$. We introduced half-points for the
purpose of proving this lemma even at the folding points of $f$.

\begin{lemma}\label{lem:atomgrowth}
Let $\boldsymbol{x}\in S$. Then
$\Delta_\psi(f(\boldsymbol{x}))=\lambda\Delta_\psi(\boldsymbol{x})$.
\end{lemma}

\begin{proof}
Consider first the case when $\boldsymbol{x}=x$ is a whole-point, i.e.
$\boldsymbol{x}\in S\setminus\tilde{Q}$. Then $x$ belongs to the
interior of some $P$-basic interval $J$. We may choose a sequence
$y_i$ in $Q\cap J$ converging to $x$ from the left-hand side, and a
sequence $z_i$ in $Q\cap J$ converging to $x$ from the right-hand
side. Then $f(y_i)$ and $f(z_i)$ are sequences in $Q$ converging to
$f(x)$ from opposite sides. By the monotonicity of $\psi$ and the
definition of $\Delta_\psi$ we have
$|\psi(z_i)-\psi(y_i)|\to\Delta_\psi(\boldsymbol{x})$ and
$|\psi(f(z_i))-\psi(f(y_i))|\to\Delta_\psi(f(\boldsymbol{x}))$. Since
$J$ is an interval of monotonicity of $f$, the result follows from
Lemma~\ref{lem:properties}~\eqref{expansion}.

Now consider the case when $\boldsymbol{x}=(x,+)$ or
$\boldsymbol{x}=(x,-)$, and suppose an appropriate half-neighborhood
of $\boldsymbol{x}$ is contained in a single $P$-basic interval $J$ so
that $f$ is either orientation-preserving or orientation-reversing at
$\boldsymbol{x}$. We may repeat the proof from the previous case, with
one modification. If $\boldsymbol{x}=(x,+)$, then we take $y_i$ to be
instead the constant sequence with each member equal to $x$. If
$\boldsymbol{x}=(x,-)$, then we take $z_i$ to be instead the constant
sequence with each member equal to $x$. Then the rest of the proof
holds as written.

Now consider the case when $\boldsymbol{x}=(x,+)$ and
$f(\boldsymbol{x})=(f(x),+)$, but every half-neighborhood $[x,x+\eps)$
meets $P$. We will show in this case that
$\Delta_\psi(\boldsymbol{x})$ and $\Delta_\psi(f(\boldsymbol{x}))$
are both zero. Choose points $z_i\in P$ which converge monotonically
to $x$ from the right and such that each $f(z_i)>f(x)$. By
continuity, $f(z_i)\to f(x)$, and after passing to a subsequence, we
may assume that this convergence is also monotone. Now we calculate
$\Delta_\psi(\boldsymbol{x})$ using the sequence $z_i$ and appealing
back to the definition of $\psi$.
\begin{equation*}
\Delta_\psi(\boldsymbol{x}) = \lim_{i\to\infty} (\psi(z_i)-\psi(x)) =
\lim_{i\to\infty} \sum_{\substack{J\in B(P)\\x<J<z_i}}v_J =
\lim_{i\to\infty} \sum_{j=i}^\infty \sum_{\substack{J\in
    B(P)\\z_{j+1}<J<z_j}} v_J = 0
\end{equation*}
The rearrangement of the sum is justified because for each $P$-basic
interval $J$ between $x$ and $z_i$ there is exactly one $j\geq i$ such
that $J$ lies between $z_{j+1}$ and $z_j$. But by
Lemma~\ref{lem:properties}~\eqref{finite}, when $i=1$ we have already
a convergent series. Thus, when we sum smaller and smaller tails of
the series, we obtain $0$ in the limit. We may apply exactly the same
argument to compute $\Delta_\psi(f(\boldsymbol{x}))$ along the
sequence $f(z_i)$, because these points also belong to the invariant
set $P$ and decrease monotonically to $f(x)$.

There are three other cases in which every appropriate
half-neighborhood of $\boldsymbol{x}$ meets $P$; again in each of
these cases $\Delta_\psi(\boldsymbol{x})=0$ and
$\Delta_\psi(f(\boldsymbol{x}))=0$ by similar arguments.
\end{proof}

The next lemma shows that the intermediate value theorem respects
our definition of half-points.

\begin{lemma}\label{lem:intermediate}
Let $x_1<x_2$ be any two points in $[0,1]$, not necessarily in $Q$,
and let $k\in\N$. Suppose that there exists a point $\boldsymbol{y}\in
S$ with $y$ strictly between $f^k(x_1)$ and $f^k(x_2)$. Then there
exists $\boldsymbol{x}\in S$ with $x$ between $x_1$ and $x_2$ such
that $f^k(\boldsymbol{x})=(\boldsymbol{y})$.
\end{lemma}

\begin{proof}
If $\boldsymbol{y}=y\in S\setminus\tilde{Q}$, we just apply the
invariance of $Q$ and the usual intermediate value theorem. If
$\boldsymbol{y}\in\tilde{Q}$, then we consider the set
$A=[x_1,x_2]\cap f^{-k}(y)$. It is nonempty by the usual intermediate
value theorem, compact by the continuity of $f^k$, and contained in
$Q$ by the invariance of $Q$. Suppose first that $f^k(x_1)<f^k(x_2)$.
If $x'$ satisfies $x_1<x'<\min A$, then $f^k(x')<y$ by the usual
intermediate value theorem and the minimality of $\min A$. It follows
that $f^k(\min A, -) = (y, -)$. Similarly, $f^k(\max A,+)=(y,+)$. Thus
$\boldsymbol{x}$ may be taken as one of the points $(\min A,-), (\max
A,+)$. The proof when $f^k(x_1)>f^k(x_2)$ is similar, except that
$f^k(\min A,-)=(y,+)$ and $f^k(\max A,+)=(y,-)$.
\end{proof}

\section{No Atoms}\label{sec:atom}

Now we are ready to solve Problem~\ref{prob:jump}.

\begin{lemma}\label{lem:noatoms}
$\psi$ has no atoms; that is, $\Delta_\psi$ is identically zero.
\end{lemma}

\begin{proof}
Assume toward contradiction that there is a point $\boldsymbol{b}\in
S$ such that $\Delta_\psi(\boldsymbol{b})>0$. For $n=0,1,2,\ldots,$
let $\boldsymbol{b}_n:=f^n(\boldsymbol{b})\in S$ and denote the
corresponding point in $[0,1]$ by $b_n$. We denote the orbit of $b$ by
$Orb(b)=\{b_0,b_1,b_2,\ldots\}$. By Lemma~\ref{lem:atomgrowth},
\begin{equation}\label{atommasses}
\Delta_\psi(\boldsymbol{b}_n) = \lambda^n \Delta_\psi(\boldsymbol{b}), \quad n\in\N
\end{equation}
and this grows to $\infty$ because $\lambda > 1$. If $Orb(b)$ has an
accumulation point in the open interval $(0,1)$, then the increment of
$\psi$ across a small neighborhood of this accumulation point is
$\infty$, contradicting Lemma~\ref{lem:properties}~\eqref{finite} and
we are done. Henceforth, we may assume that the orbit of $b$ only
accumulates at (one or both) endpoints of $[0,1]$. Consider first the
case when $Orb(b)$ accumulates at only one endpoint of $[0,1]$, and
assume without loss of generality that $\lim_{n\to\infty}b_n=1$.

Since $f$ is mixing, it must have a fixed point $w$ with $0<w<1$.
Since $b_n\to 1$, it follows that $b_n>w$ for all sufficiently large
$n$. Thus, after replacing $\boldsymbol{b}$ and $b$ with their
appropriate images, we may assume that $b_n>w$ for all $n\in\N$.
Equation~\eqref{atommasses} continues to hold, and it follows that $b$
is not a fixed point for $f$, so $b\neq 1$.

Now consider the following claim:
\begin{multline}\label{star}
\textnormal{For all } N\in\N \textnormal{ there exist } n>N
\textnormal{ and } \boldsymbol{a}\in S \\ \textnormal{ such that }
a\notin Orb(b),\ f(\boldsymbol{a})=\boldsymbol{b}_n,
\textnormal{ and } w<a<b_{n+1}.\tag{$\star$}
\end{multline}

The proof of claim~\eqref{star} proceeds in two cases. First, assume
that $b_N<b_{N+1}<b_{N+2}<\ldots$; i.e., starting from time $N$, the
orbit of $b$ moves monotonically to the right. Since $f$ is mixing,
the interval $[b_{N+1},1]$ cannot be invariant, so there must exist
$c>b_{N+1}$ with $f(c)<b_{N+1}$. Take $n=\max\{i:b_i<c\}$. Clearly
$n>N$. The relevant ordering of points is $b_{n-1}<b_n<c<b_{n+1}$.
Since $f(b_n)>b_n$ and $f(c)<b_n$, it follows by
Lemma~\ref{lem:intermediate} that there exists $\boldsymbol{a}$ with
$a$ between $b_n$ and $c$ such that
$f(\boldsymbol{a})=\boldsymbol{b}_n$. Clearly, $a\neq b_{n-1}$. It
follows that $a\notin Orb(b)$. Moreover, $w<a<b_{n+1}$.

The remaining case is that there exists $i\geq N$ such that
$b_{i+1}<b_i$; i.e., at some time later than $N$, the orbit moves to
the left. But our orbit is converging to the right-hand endpoint of
$[0,1]$, so it cannot go on moving to the left forever. Let
$n=\min\{j>i:b_{j+1}>b_j\}$. We have $n>N$, and the relevant ordering
of points is $b_{n-1}>b_n$ and $b_{n+1}>b_n$. Since $f(w)<b_n$ and
$f(b_n)>b_n$, it follows by Lemma~\ref{lem:intermediate} that there
exists $\boldsymbol{a}$ with $a$ between $w$ and $b_n$ such that
$f(\boldsymbol{a})=\boldsymbol{b}_n$. Again, we see that $a\neq
b_{n-1}$, so $a\notin Orb(b)$. Finally, $a<b_{n+1}$. This concludes
the proof of claim~\eqref{star}.

Now we apply claim~\eqref{star} recursively to find infinitely many
distinct atoms between $w$ and $b$, each with the same positive mass.
At stage 1, find $n_1$ and $\boldsymbol{a}_1$ with $a_1\notin Orb(b)$
such that $f(\boldsymbol{a}_1)=\boldsymbol{b}_{n_1}$ and
$w<a_1<b_{n_1+1}$. Now we apply Lemma~\eqref{lem:intermediate} to
$f^{n_1+1}$ to find $\boldsymbol{x}_1$ with $x_1$ between $w$ and $b$
such that $f^{n_1+1}(\boldsymbol{x}_1)=\boldsymbol{a}_1$. Then
$f^{n_1+2}(\boldsymbol{x}_1)=\boldsymbol{b}_{n_1}$, so by applying
Lemma~\ref{lem:atomgrowth} and equation~\eqref{atommasses} we have
$\Delta_\psi(\boldsymbol{x}_1) =
\lambda^{-(n_1+2)}\Delta_\psi(\boldsymbol{b}_{n_1}) =
\lambda^{-2}\Delta_\psi(\boldsymbol{b})$. The point $\boldsymbol{x}_1$
will serve as the first of infinitely many points between $w$ and $b$
at which $\psi$ has this particular increment. At stage $i$, set
$N=n_{i-1}$ and apply claim~\eqref{star} to find $n_i$ and
$\boldsymbol{a}_i$ with $n_i>n_{i-1}$. Again, we can find
$\boldsymbol{x}_i$ with $x_i$ between $w$ and $b$ and
$f^{n_i+1}(\boldsymbol{x}_i)=\boldsymbol{a}_i$, whence
$\Delta_\psi(\boldsymbol{x}_i)=\lambda^{-2}\Delta_\psi(\boldsymbol{b})$
as before. It remains to check that the points $\{x_i\}$ are distinct.
Observe that $f^{n_i+1}(x_i)=a_i$ does not belong to the invariant set
$Orb(b)$, whereas $f^{n_i+2}(x_i)=b_{n_i}\in Orb(b)$. By construction,
the numbers $\{n_i\}$ are all distinct. Thus, the points $\{x_i\}$ are
distinguished from one another by the time required to make first
entrance into $Orb(b)$.

Now we use our atoms to produce a contradiction. By
Lemma~\ref{lem:properties}~\eqref{finite}, the increment
$\psi(b)-\psi(w)$ is finite. Choose an integer $n$ large enough that
$n\lambda^{-2}\Delta_\psi(\boldsymbol{b})>\psi(b)-\psi(w)$. Consider
the points $\boldsymbol{x}_1, \boldsymbol{x}_2, \ldots,
\boldsymbol{x}_n$, and let $\delta$ be the minimum distance between
two adjacent points of the set $\{w,b\}\cup\{x_1,x_2,\ldots,x_n\}$.
For each $i=1,\ldots,n$ there exist $y_i, z_i \in Q$ with
$y_i<x_i<z_i$ and $\max\{z_i-x_i,x_i-y_i\}<\delta/2$. Then
$\psi(z_i)-\psi(y_i) \geq \lambda^{-2}\Delta_\psi(\boldsymbol{b})$. By
the monotonicity of $\psi$,
\begin{equation*}
\psi(b)-\psi(w)\geq\sum_{i=1}^n \psi(z_i)-\psi(y_i) >
n\lambda^{-2}\Delta_\psi(\boldsymbol{b})>\psi(b)-\psi(w).
\end{equation*}
This is a contradiction; in words, we cannot have infinitely many
atoms between $w$ and $b$ all having the same positive mass when the
total increment of $\psi$ between $w$ and $b$ is finite. This
completes the proof in the case that $Orb(b)$ accumulates at only one
endpoint of $[0,1]$.

Finally, let us say a few words about the case when $Orb(b)$
accumulates at both endpoints of $[0,1]$. In this case, $f(0)=1$ and
$f(1)=0$ by continuity. Again by continuity, for sufficiently large
$n$ the points $b_n$ belong alternately to a small neighborhood of $0$
and a small neighborhood of $1$. Thus, the subsequence $b_{2n}$
accumulates only on a single endpoint of $[0,1]$. The map $f^2$ is
again topologically mixing. It is straightforward, then, to modify the
above proof to deal with this case, by working along the subsequence
$b_{2n}$ and writing $f^2$ and $\lambda^2$ in place of $f$ and
$\lambda$.
\end{proof}

\section{The Rest of the Proof of Theorem
 ~\ref{thm:main}}\label{sec:finish}

Having resolved Problem~\ref{prob:jump}, we are ready to finish the
proof of Theorem~\ref{thm:main}.

\begin{proof}
It remains to show that condition~\eqref{criterion} is sufficient. We
have defined on the dense subset $Q\subset[0,1]$ a strictly monotone
map $\psi:Q\to[-\infty,\infty]$. In light of Lemma~\ref{lem:noatoms},
the formula $\psi(x)=\sup \psi(Q\cap[0,x)) = \inf \psi(Q\cap(x,1])$
gives a well-defined extension $\psi:[0,1]\to[-\infty,\infty]$. Strict
monotonicity of the extension follows from the strict monotonicity of
$\psi|_Q$ and the density of $Q$. We claim that the extended function
$\psi$ is continuous. It suffices to verify for each $x$ that
$\psi(x)=\lim_{y\to x^-}\psi(y)=\lim_{z\to x^+}\psi(z)$. By
monotonicity of $\psi$ and the density of $Q$ we may evaluate these
one-sided limits using points $y,z\in Q$, and by our definition of the
extended map $\psi$ the claim follows. Finally, from strict
monotonicity and continuity, it follows that
$\psi:[0,1]\to[-\infty,\infty]$ is a homeomorphism onto its image.

Define a map $g:\psi([0,1])\to\psi([0,1])$ by the composition
$g:=\psi\circ f\circ \psi^{-1}$. It is countably piecewise monotone
and Markov with respect to $\psi(P)$. If $y=\psi(x)$ and $y'=\psi(x')$
belong to a single $\psi(P)$-basic interval, then $x$ and $x'$ belong
to an interval of monotonicity of $f$. By
Lemma~\ref{lem:properties}~\eqref{expansion} and the density of $Q$ we
may conclude that $|g(y)-g(y')|=\lambda|y-y'|$. This shows that $g$
has constant slope $\lambda$.
\end{proof}

\section{Constant Slope on the Interval}\label{sec:example}

We present now a map $f:[0,1]\to[0,1]$, $f\in\CPMM$, with the following
linearizability properties. For any $\lambda\geq\lambda_{min}$, where
$\lambda_{min}$ is the positive real root of
$\lambda^3-2\lambda^2-\lambda-2$ (approximately 2.66), there is a map
$g:[0,1]\to[0,1]$ of constant slope $\lambda$ conjugate to $f$.
Moreover, the topological entropy of $f$ is equal to $\log
\lambda_{min}$. However, $f$ is not conjugate to any map of constant
slope on the extended real line or the extended half line. This
sharply illustrates the point that for countably piecewise monotone
maps, constant slope gives only an upper bound for topological
entropy.

Bobok and Soukenka~\cite{BS} have constructed a map with similar
linearizability properties, that is, with entropy $\log 9$ and with
conjugate maps of every constant slope $\lambda\geq9$. However, their
example exhibits transient Markov dynamics~\cite{BB}, whereas our map
$f$ exhibits strongly positive recurrent Markov dynamics (in the sense
of the Vere-Jones recurrence hierarchy for countable Markov chains,
see~\cite{Ru, VJ}). We regard this as evidence that the existence of
constant slope models for a given map is in some part independent of
the recurrence properties of the associated Markov dynamics (but see
the discussion at the beginning of section~\ref{sec:example2}).

To construct the map $f$, we subdivide the interval $[0,1]$ into
countably many subintervals $\{A_i\}_{i=0}^\infty$,
$\{B_i\}_{i=0}^\infty$, $\{C_i\}_{i=0}^\infty$, and $D$, ordered from
left to right as follows:
\begin{equation*}
D < C_0 < C_1 < C_2 < \ldots < x_{fixed} < \ldots < B_2 < A_2 < B_1 <
A_1 < B_0 < A_0.
\end{equation*}
We specify the lengths of the intervals $A_i$, $B_i$, $C_i$, $D$ to be 
equal (respectively) to the numbers $a_i$, $b_i$, $c_i$, $d$
given in equation~\eqref{fEig} below, taking $\lambda=\lambda_{min}$.
The partition $P$ consists of the endpoints of these intervals together
with their (unique) accumulation point, which we denote $x_{fixed}$ 
and set as a fixed point for $f$.
We prescribe for $f$ the following Markov dynamics:
\begin{gather*}\label{fMarDyn}
f(D)=[0,1], \quad f(C_0)=D, \quad f(C_i)=C_{i-1}, i\geq1,
\\ f(A_i)=\left(\bigcup_{j=i+1}^\infty A_j\right) \cup
\left(\bigcup_{j=i+1}^\infty B_j\right) \cup \left(\bigcup_{j=i}^\infty
C_j\right) \\ f(B_i)=\left(\bigcup_{j=i+2}^\infty A_j\right) \cup
\left(\bigcup_{j=i+2}^\infty B_j\right) \cup \left(\bigcup_{j=i}^\infty
C_j\right)
\end{gather*}
Moreover, we prescribe that our map will increase linearly on each of
the intervals $A_i$, $C_i$, and decrease linearly on each of the
intervals $B_i$, $D$. This completes the definition of $f$; we present its 
graph in Figure~\ref{fig:Michal}. As we will see, the lengths we chose
for the $P$-basic intervals comprise an eigenvector for the Markov transition matrix
with eigenvalue $\lambda_{min}$.  Thus, by construction, $f$ has constant slope
$\lambda=\lambda_{min}>2$.  This allows us to verify that $f$ is 
topologically mixing.  Indeed, let $U,V$ be a pair of arbitrary open intervals.
The iterated images $f^{n}(U)$ grow in size until some image contains 
an entire $P$-basic interval (any interval which does not contain an entire 
$P$-basic interval is folded by $f$ in at most one place, so that its image 
grows by a factor of at least $\lambda_{min}/2>1$, and such growth cannot 
continue indefinitely in a finite length state space).  Consulting the Markov 
transition diagram in Figure~\ref{fig:Michal2}, we see that the union of 
images of any given $P$-basic interval includes all $P$-basic intervals and
therefore intersects the open set $V$.

\begin{figure}[htb!]
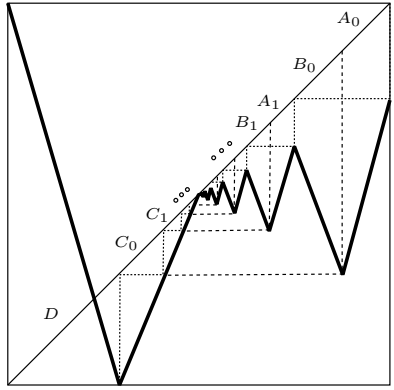
\caption{A map conjugate to maps of any constant slope $\lambda \geq
  \lambda_{min}$.}\label{fig:Michal}
\end{figure}

Next we consider the nonnegative eigenvalues and eigenvectors
corresponding to our Markov partition. That is, we solve the system of
equations
\begin{equation}\label{fSys}
\left\{
\begin{array}{lcl}
\lambda a_i &=& \sum_{j=i+1}^\infty a_j + \sum_{j=i+1}^\infty b_j +
\sum_{j=i}^\infty c_j \\ \lambda b_i &=& \sum_{j=i+2}^\infty a_j +
\sum_{j=i+2}^\infty b_j + \sum_{j=i}^\infty c_j \\ \lambda c_i &=&
c_{i-1} \\ \lambda c_0 &=& d \\ \lambda d &=& \sum_{j=0}^\infty a_j +
\sum_{j=0}^\infty b_j + \sum_{j=0}^\infty c_j + d
\end{array}
\right.
\end{equation}
where $a_i$, $b_i$, $c_i$, and $d$ represent the entries corresponding
to the intervals $A_i$, $B_i$, $C_i$, and $D$, respectively, and
$\lambda$ represents the eigenvalue. We subtract the first equation of
\eqref{fSys} with index $i+1$ from the first equation with index $i$
to obtain
\begin{equation}\label{step1}
\lambda(a_i-a_{i+1})=a_{i+1}+b_{i+1}+c_i.
\end{equation}
Then we subtract the first equation of~\eqref{fSys} with index $i+2$
from the second equation with index $i+1$ to obtain
\begin{equation}\label{step2}
b_{i+1}=a_{i+2}+\lambda^{-1}c_{i+1}.
\end{equation}
Substituting~\eqref{step2} into~\eqref{step1}, applying the third
equation from~\eqref{fSys} to express all $c_i$'s in terms of $c_0$,
and rearranging terms, we obtain
\begin{equation}\label{step3}
0=a_{i+2}+(\lambda+1)a_{i+1}-\lambda a_i +
\lambda^{-i}(1+\lambda^{-2})c_0
\end{equation}
This equation defines a nonhomogeneous, constant-coefficient linear
recurrence relation on the terms $a_i$. The theory of linear
recurrence equations tells us that the general solution
to~\eqref{step3} is
\begin{equation}\label{step4}
a_i=\alpha x_+^i + \beta x_-^i +
\frac{(\lambda^2+1)c_0}{\lambda^3-\lambda^2-\lambda-1} \lambda^{-i},
\end{equation}
where $\alpha$ and $\beta$ are arbitrary constants and $x_+$, $x_-$
are the positive and negative solutions of the characteristic equation
$x^2+(\lambda+1)x-\lambda=0$ (they are real because $\lambda>0$).
Observe that the terms $c_i$ grow
exponentially with rate $\lambda^{-1}$, and from the first equation
of~\eqref{fSys} we have $\sum_{i=0}^\infty c_i < \lambda a_0 <
\infty$. It follows that $\lambda > 1$. Now of the three exponential
terms in~\eqref{step4}, the base with the greatest modulus is
$x_-<-1$. It follows that we must take $\beta=0$ to achieve
nonnegativity of the terms $a_i$. For $\lambda$ between $1$ and the
real root of $\lambda^3-\lambda^2-\lambda-1$ we have simultaneously
(miracle) that $\lambda^{-1}>x_+$
and that the coefficient of the $\lambda^{-i}$ term
is negative. Nonnegativity of the terms $a_i$ forces us to consider
only $\lambda$ greater than the real root of
$\lambda^3-\lambda^2-\lambda-1$, and henceforward we may assume that
$\lambda^3-\lambda^2-\lambda-1>0$ and (what is equivalent) that
$x_+>\lambda^{-1}$. Now that we are equipped with
equations~\eqref{step2},~\eqref{step4}, and the third equation
of~\eqref{fSys}, we are able to sum the geometric series in the fifth
equation of~\eqref{fSys}, which gives us that
$\alpha=\frac{x_+\lambda(\lambda^3-2\lambda^2-\lambda-2)}
{\lambda^3-\lambda^2-\lambda-1}c_0$. Again, we invoke nonnegativity of
the terms $a_i$ to conclude that $\lambda$ must be greater than or
equal to the real root of $\lambda^3-2\lambda^2-\lambda-2$, which is
approximately $2.66$. Combining all of our results so far, using the
equality $\frac{x_++1}{x_+-1}=\frac{\lambda}{x_+}$, and
choosing a scaling constant to clear all denominators, we have that
any solution to the system~\eqref{fSys} must be of the form
\begin{equation}\label{fEig}
\left\{
\begin{array}{l}
\lambda\in[\lambda_{min},\infty), \quad \lambda_{min} =
  \textnormal{the real solution of }\lambda^3-2\lambda^2-\lambda-2=0
  \\ x= \textnormal{the positive solution of }
  x^2+(\lambda+1)x-\lambda=0 \\ \alpha = x \lambda (\lambda^3 -
  2\lambda^2 - \lambda - 2) \\ a_i = \alpha x^i +
  (\lambda^2+1)\lambda^{-i} \\ b_i = \alpha x^{i+1} +
  (\lambda^2-1)\lambda^{-i} \\ c_i =
  (\lambda^3-\lambda^2-\lambda-1)\lambda^{-i} \\ d =
  (\lambda^3-\lambda^2-\lambda-1)\lambda
\end{array}
\right.
\end{equation}
Conversely, we can verify that equation~\eqref{fEig} does indeed
define a nonnegative solution to~\eqref{fSys}. This completes our
eigenvector analysis. In light of Theorem~\ref{thm:main} this
establishes our claims about the existence of constant slope maps
conjugate to $f$ for each $\lambda\geq\lambda_{min}$.

It is worth noting that the transition matrix $T$ cannot have
unsummable nonnegative eigenvectors for the simple reason that the
$P$-basic interval $D$ contains in its image all $P$-basic intervals,
so that the sum of the entries of any eigenvector must be
$\lambda\cdot d<\infty$.

Next, we wish to argue that the topological entropy of $f$ is equal 
to $\log\lambda_{min}$. We begin by recalling the necessary facts 
from the theory of transitive countable Markov chains. The \emph{Perron value}
$\lambda_M$ of a transitive, countable state Markov
chain is defined~\cite{VJ67} by $\lambda_M=\lim (p_{uu}^{(n)})^{1/n},$ 
where $p_{uu}^{(n)}$ denotes the number of length $n$ loops in the 
chain's transition graph which start and end at a fixed, arbitrary vertex 
$u$; the limit is independent of the choice of $u$. In contrast, the numbers 
$f_{uu}^{(n)}$ count only the length $n$ \emph{first-return} loops, 
which start and end at the vertex $u$ but do not visit $u$ at any 
intermediate time.  It may happen that $\Phi_u:=\limsup 
(f_{uu}^{(n)})^{1/n} < \lambda_M$ for some vertex $u$; then the same 
inequality holds for every vertex $u$ and the chain is called \emph{
strongly positive recurrent}~\cite[Definition 2.3 and Theorem 2.7]{Ru}.  
Moreover,~\cite[Proposition 2.4]{Ru} gives the following equivalence,
which allows us to detect strongly positive recurrence:
\begin{equation}\label{spr}
\Phi_u<\lambda_M \mbox{ if and only if }\sum_{n\geq1} f_{uu}^{(n)}\Phi_u^{-n}>1.
\end{equation}
Strongly positive recurrent chains are a special case of recurrent chains, 
for which the Perron value is known to be equal to the minimum of the set 
of eigenvalues for nonnegative eigenvectors~\cite[Theorem 2]{Pr}. The
connection to interval maps is given by~\cite[Proposition 7]{BB}, which 
says that the entropy of a topologically mixing countably piecewise 
monotone and Markov map is given by the logarithm of the Perron 
value of the corresponding Markov chain.

Consider now the countable state topological Markov chain
associated to our particular map $f$ with its given Markov partition. The
transition diagram of this chain is shown in Figure~\ref{fig:Michal2}.
Denoting its Perron value by $\lambda_M$ and applying the results
of the preceding paragraph, we have $h_{top}(f)=\log\lambda_M$.
To show that $\lambda_M=\lambda_{min}$, it suffices to count first return
paths and prove that the chain is strongly positive recurrent.
\begin{figure}[htb!]
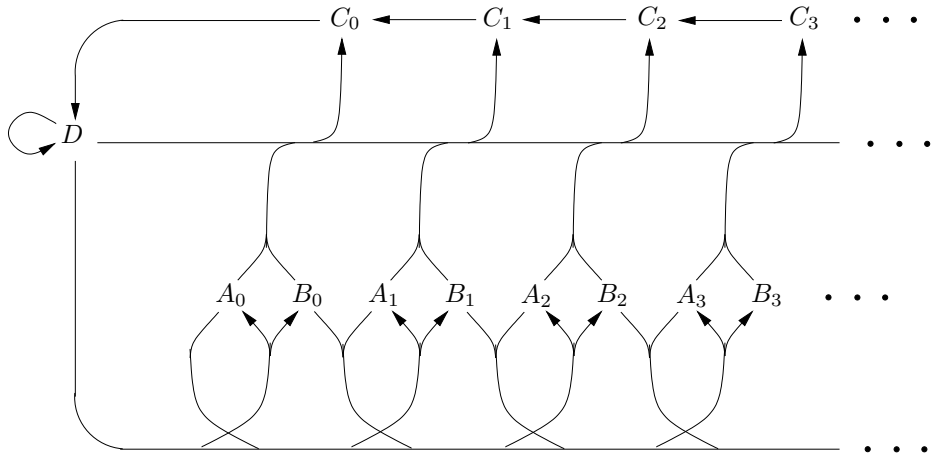
\caption{The transition diagram.}\label{fig:Michal2}
\end{figure}

Let $f^{(n)}$ denote the number of first return paths of length $n$
from the vertex $D$ to itself. To compute these numbers, we organize
the collection of all first return paths from $D$ to itself as
follows. Using the convention $D=C_{-1}$, we find (see Figure~\ref{fig:Michal2}) that each first
return path may be written uniquely in the form $DxC_nC_{n-1}\cdots
C_0D$, where $x$ is a string (perhaps empty) consisting only of $A$'s
and $B$'s and $n\geq-1$. If the final symbol in the string $x$ is not $B_n$, then we
declare that $DxC_nC_{n-1}\cdots C_0D$ has three descendants, namely,
$DxC_{n+1}C_{n}\cdots C_0 D$, $DxA_{n+1}C_{n+1}C_{n}\cdots C_0 D$, and
$DxB_{n+1}C_{n+1}C_{n}\cdots C_0 D$. But if the final symbol in the
string $x$ is $B_n$, then we declare that $DxC_nC_{n-1}\cdots C_0D$ has
only one descendant, namely, $DxC_{n+1}C_{n}\cdots C_0 D$. This
relationship organizes the set of first return paths into a tree
(Figure~\ref{fig:tree}), in which each first return path traces a
unique ancestry back to the shortest first return path $DD$. Moreover,
we may organize this tree into levels, corresponding to the
lengths of the first return paths.
\begin{figure}[htb!]
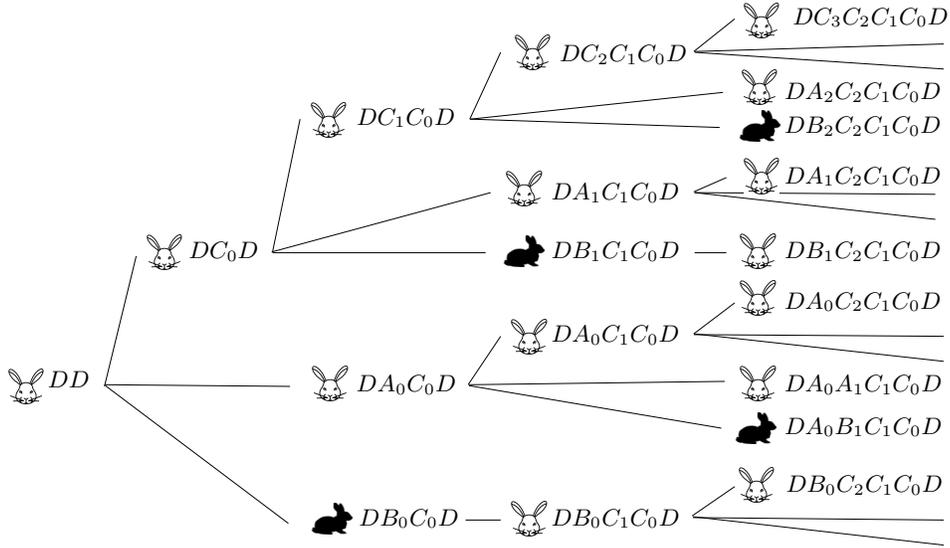
\caption{The tree of first return paths.}\label{fig:tree}
\end{figure}

Looking again at Figure~\ref{fig:tree}, we see that the problem of
computing the growth rate of the numbers $f^{(n)}$ is the same as
computing the growth rate of a population of white and black rabbits,
reproducing according to the rules that each white rabbit gives birth
to a white rabbit on its first birthday and to twin white and black
rabbits on its second birthday, whereas a black rabbit gives birth to
a white rabbit on its first birthday and no additional rabbits. Thus,
the sub-population of white rabbits is growing according to the
recurrence relation $w^{(n+3)}=w^{(n+2)}+w^{(n+1)}+w^{(n)}$, while the
total population at generation $n$ is $f^{(n)}=w^{(n)}+w^{(n-2)}$.
This yields the closed form expressions
\begin{equation*}
\begin{gathered}
w^{(n)}=\alpha x_1^n+\beta x_2^n+\gamma x_3^n, \\
f^{(n)}=\alpha(1+x_1^{-2})x_1^n+\beta(1+x_2^{-2})x_2^n+\gamma(1+x_3^{-2})x_3^n,
\end{gathered}
\end{equation*}
where $x_1,x_2,x_3$ are the three roots of the characteristic
polynomial $x^3-x^2-x-1=0$ and the coefficients $\alpha$, $\beta$,
$\gamma$ can be determined by fitting the initial data $w^{(1)}=1$,
$w^{(2)}=1$, $w^{(3)}=2$. Of these three roots we have $x_1 \approx
1.84$ real and $x_2$, $x_3$ complex conjugates with modulus less than
$1$. Now the simple observation that $w^{(3)}>w^{(2)}$ gives us that
$\alpha\neq0$, and therefore we obtain the limit
$(f^{(n)})^{1/n} \to x_1$. However, the sum $\sum f^{(n)} x_1^{-n}$ diverges,
because the terms $f^{(n)}x_1^{-n}$ are converging to the nonzero constant
$\alpha(1+x_1^{-2})$. Comparing with equation~\eqref{spr}, we see that
our chain is strongly positive recurrent, which is what we wanted to show.

\section{Constant Slope on the Extended Real Line and Half Line}
\label{sec:extended}

We present now a map $f:[0,1]\to[0,1], f\in\CPMM$, with the following
linearizability properties. It is conjugate to maps of constant slope $\lambda$ 
on the extended real line (respectively, extended half line) for every 
$\lambda\geq2+\sqrt5$ (respectively, $\lambda>2+\sqrt5$).

First, define a map $F:\mathbb{R}\to\mathbb{R}$ as the piecewise
affine ``connect-the-dots'' map with ``dots'' at $(k,k-1)$,
$(k+b,k+b+1)$, $k\in\Z$, where $b=(\sqrt5-1)/2$; it is piecewise
monotone and Markov with respect to the set $\{k,k+b:k\in\Z\}$, and it
has constant slope $2+\sqrt5$. Moreover, fix a
homeomorphism $h:(0,1)\to\mathbb{R}$; if we wish to be concrete, we
may take $h(x)=\ln(x/(1-x))$. Let $f:[0,1]\to[0,1]$ be the map
$h^{-1}\circ F\circ h$ with additional fixed points at $0$, $1$. Then
$f$ is piecewise monotone and Markov with respect to the set
$P=\{0,1\}\cup\{h^{-1}(k),h^{-1}(k+b):k\in\Z\}$.  Figure~\ref{fig:F} shows the
graphs of $F$ and $f$ together with their Markov partitions.

\begin{figure}[hbt!]
\centering
\includegraphics[scale=.285]{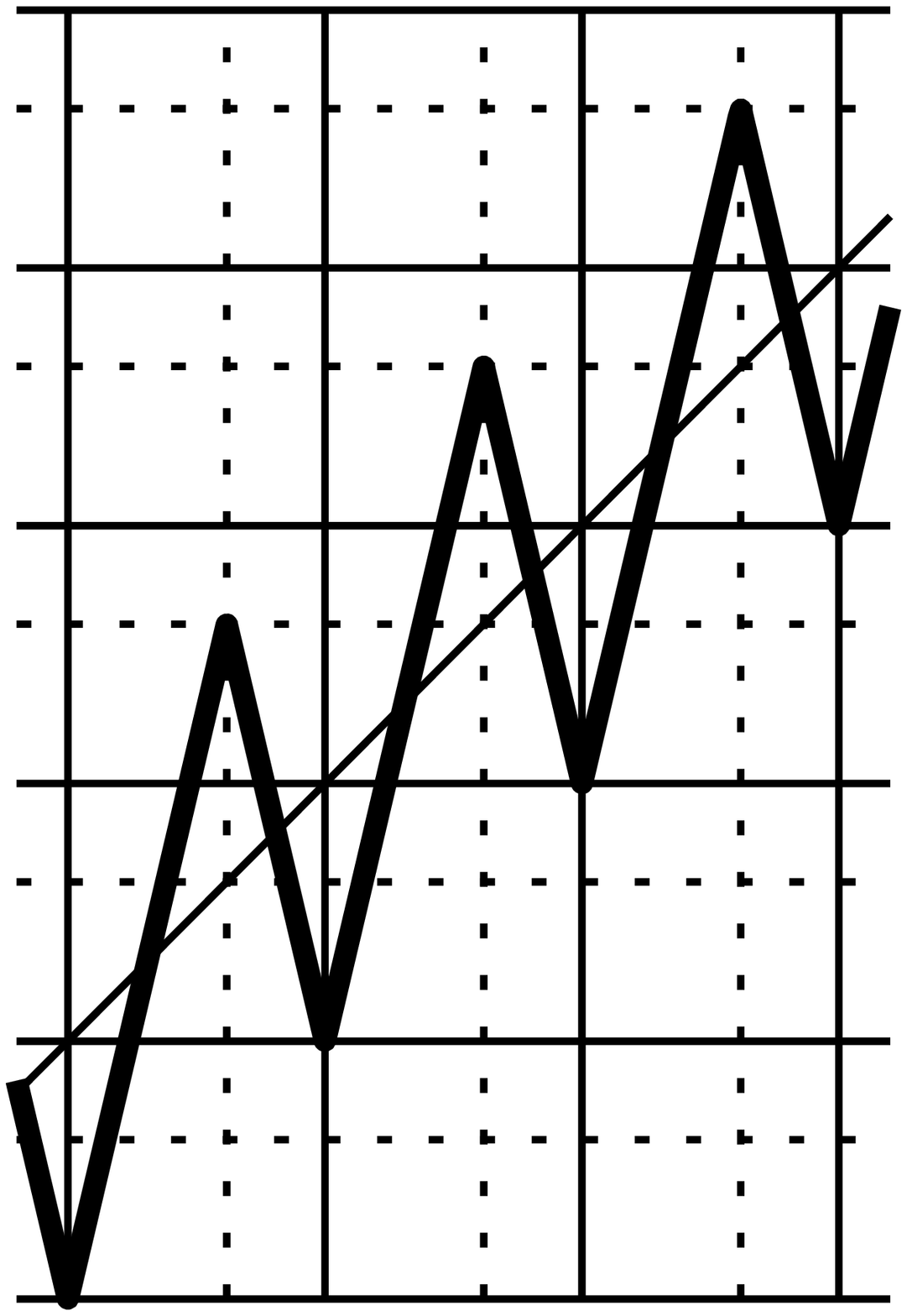}
\hspace{2cm}
\includegraphics[scale=.285]{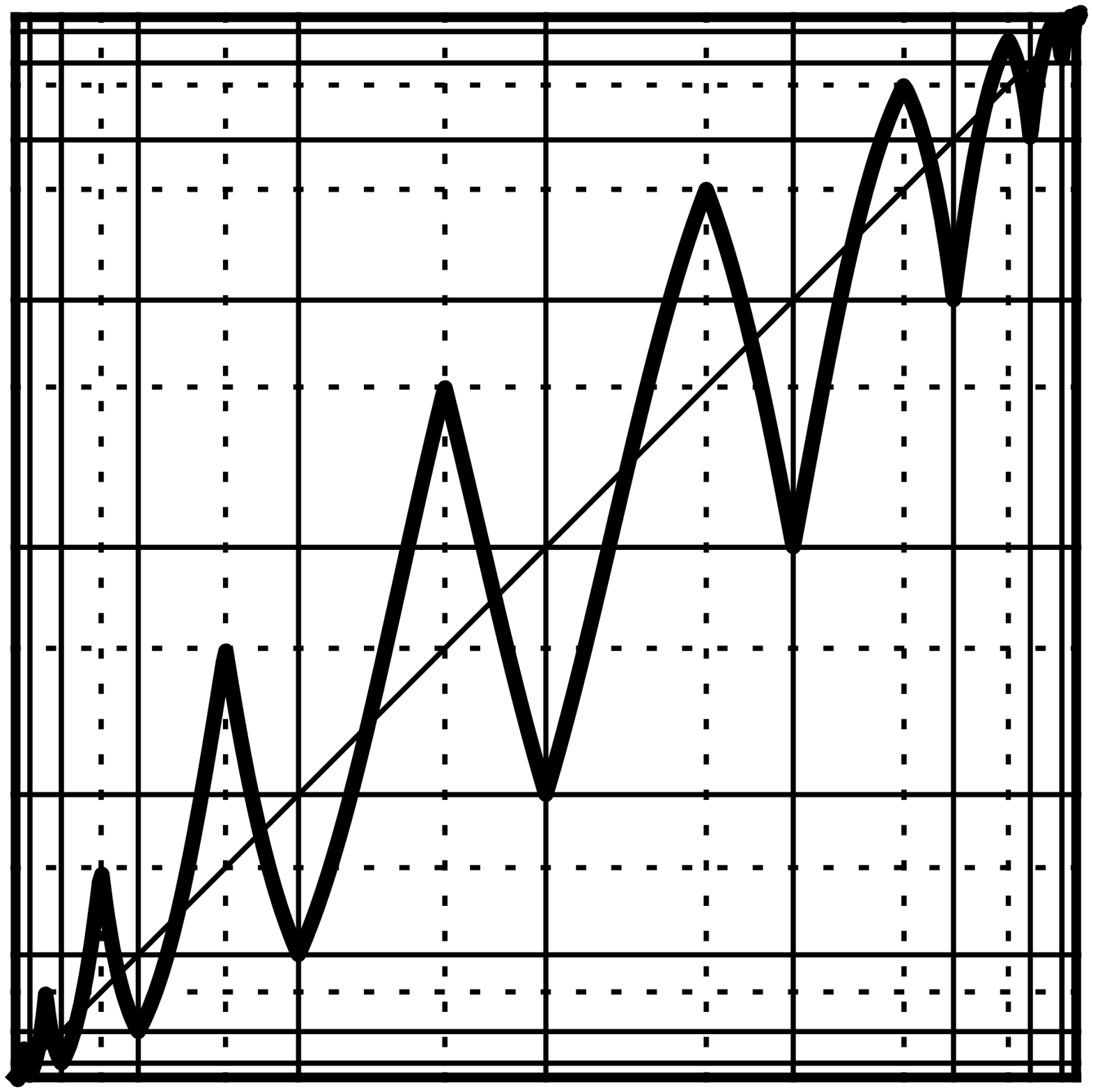}
\caption{The constant slope map $F:\R\to\R$ and the conjugate 
map $f:[0,1]\to[0,1]$ with fixed points at $0,1$.}\label{fig:F}
\end{figure}

We enumerate the $P$-basic intervals as follows:
\begin{equation}\label{enum}
I_{2k}=[h^{-1}(k),h^{-1}(k+b)], \quad
I_{2k+1}=[h^{-1}(k+b),h^{-1}(k+1)], \quad k\in\Z.
\end{equation}
The Markov transitions are given by
\begin{equation}\label{MarTra}
f(I_{2k})=\bigcup\limits_{i=2k-2}^{2k+2}I_i, \quad
f(I_{2k+1})=\bigcup\limits_{i=2k}^{2k+2}I_i, \quad
k\in\Z.
\end{equation}

We must verify that $f$ is mixing.  Let $U,V$ be any pair of nonempty
open intervals.  We may assume that the points $0,1$ do not belong to 
$U$ or $V$.  Passing through the conjugacy, we may work instead 
with the map $F$ and the intervals $h(U), h(V)$.  Each iterated image of 
$h(U)$ that contains at most one folding point of $F$ is expanded in length by 
a factor of at least $\frac12(2+\sqrt{5})>1$.  Therefore some iterated image 
of $h(U)$ contains an entire $h(P)$-basic interval.  The Markov 
transitions are such that the iterated images of an arbitrary basic interval
eventually include any other given basic interval.  This establishes the mixing
property. 

Let $T$ be the 0-1 transition matrix for the map $f$ and the partition set $P$.
In light of Theorem~\ref{thm:main}, we wish to
find all nonnegative solutions $v\in\R^{\B(P)}, \lambda>1$ to the equation
$Tv=\lambda v$. Comparing equation~\eqref{MarTra} with the definition of $T$,
we are looking for all nonnegative solutions to the infinite system of
equations
\begin{equation}\label{InfSys}
\def\arraystretch{1.5}
\left\{
\begin{array}{lcl}
\lambda\, v_{I_{2k}}&=&\sum_{i=2k-2}^{2k+2}v_{I_i} \\
\lambda\, v_{I_{2k+1}}&=&\sum_{i=2k}^{2k+2}v_{I_i} \\
\end{array}
\right.
\qquad \qquad k\in\Z
\end{equation}
Adding and subtracting equations, we obtain
\begin{equation*}
\left\{
\def\arraystretch{1.5}
\begin{array}{l}
\lambda(v_{I_{2k+1}}+v_{I_{2k-1}}-v_{I_{2k}})=v_{I_{2k}}\\
\lambda v_{I_{2k+1}}=v_{I_{2k}}+v_{I_{2k+1}}+v_{I_{2k+2}}
\end{array}
\right.
\qquad \qquad k\in\Z.
\end{equation*}
Solving for later variables in terms of earlier ones, we obtain
\begin{equation}\label{InfSysLinRec}
\def\arraystretch{1.5}
\left[
\begin{array}{c}
v_{I_{2k+1}} \\
v_{I_{2k+2}}
\end{array}
\right]
=
\underbrace{
\left[
\begin{array}{cc}
-1 & 1+\frac1\lambda \\
-\lambda+1 & \lambda-1-\frac1\lambda
\end{array}
\right]
}
_{M_\lambda}
\left[
\begin{array}{c}
v_{I_{2k-1}} \\
v_{I_{2k}}
\end{array}
\right]
\qquad \qquad k\in\Z.
\end{equation}
Equation~\eqref{InfSysLinRec} should be regarded as a linear
recurrence relation on $v$.  Notice (miracle) that $\det M_\lambda=1$ 
is independent of $\lambda$.  Using the invertibility of $M_\lambda$, 
we may conclude inductively that
\begin{equation*}
\def\arraystretch{1.5} \left[\begin{array}{c} v_{I_{2k+1}}
    \\ v_{I_{2k+2}} \end{array}\right] = M_\lambda^k \left[\begin{array}{c}
    v_{I_1} \\ v_{I_2} \end{array} \right], \quad k\in\Z.
\end{equation*}
We may regard the matrix $M_\lambda$ as defining a dynamical
system on $\R^2$. Then the entries of $v$ are the orbit of the initial
point $(v_{I_1},v_{I_2})$. To obtain nonnegative entries for $v$, we
must choose the initial point so that the whole orbit (both forward 
and backward) remains in the first quadrant.  We can solve this problem 
using the elementary theory of linear transformations on 
$\mathbb{R}^2$.  There are 3 cases we must consider -- they are 
pictured in Figure~\ref{fig:shear} and explained in the following three
paragraphs.

\begin{figure}[htb!]
\input{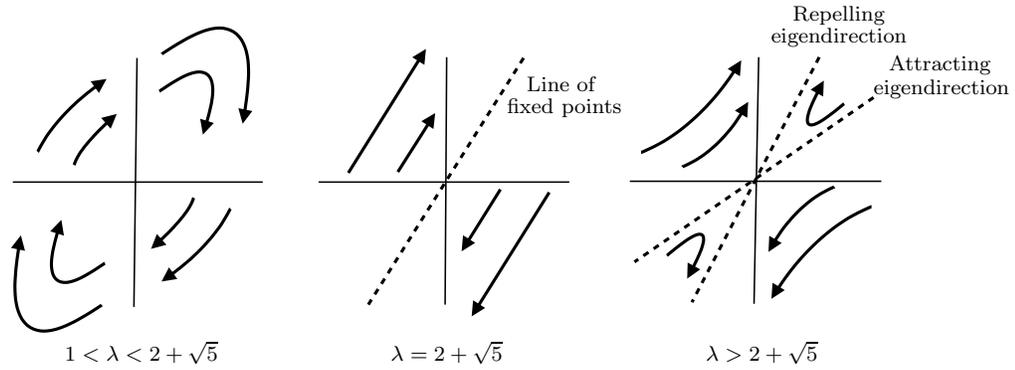}
\vspace{-0.5cm}
\caption{The action of $M_\lambda$ on $\mathbb{R}^2$}\label{fig:shear}.
\end{figure}

If $1<\lambda<2+\sqrt5$, then the eigenvalues of $M_\lambda$ are 
complex conjugates and $M_\lambda$ ``rotates'' $\mathbb{R}^2$ about 
the origin.  In this case, no orbit stays in the first quadrant.  The diligent 
reader may verify the implications $v_{I_1}\geq v_{I_2}\geq0
\implies v_{I_4}<0$ and $v_{I_2}\geq v_{I_1}\geq0 \implies v_{I_{-1}}<0$.

If $\lambda=2+\sqrt5$, then $M_\lambda$ has unique real eigenvalue 
$1$ with algebraic multiplicity 2 and geometric multiplicity 1.  Thus, 
$M_\lambda$ acts as a shear on $\mathbb{R}^2$ parallel to a line of 
fixed points (corresponding to the unique eigenvector of $M_\lambda$).
The only way to obtain a whole orbit in the first quadrant is to choose the
initial point $(v_{I_1},v_{I_2})$ from the line of fixed points.  This yields
(up to a scalar multiple) the unique nonnegative solution
\begin{equation}\label{InfSysSol}
v_{I_{2k}}=2, \quad v_{I_{2k+1}}=\sqrt5-1, \quad k\in\Z
\end{equation} 
Applying Theorem~\ref{thm:main}, we recover (up to 
scaling, and with fixed points at $\pm\infty$) the constant slope map $F$ 
which started our whole discussion.

If $\lambda>2+\sqrt5$, then $M_\lambda$ has distinct positive, real 
eigenvalues whose product is $1$.  There are distinct eigenvectors in 
the first quadrant and the origin is a saddle fixed point.  Any initial point 
$(v_{I_1},v_{I_2})$ chosen between these eigendirections in the first 
quadrant yields an unsummable, nonnegative $v$.  We can achieve 
unsummability of $v$ on one side or on both sides, according as we 
choose the initial point to lie on one of these eigendirections or strictly 
between them.  Accordingly, Theorem~\ref{thm:main} yields a constant 
slope map either on an extended half line or on the extended real line,
(see Remark~\ref{rem:gPhaseSpace}).

\section{No Constant Slope}\label{sec:example2}

We construct now a
topologically mixing map $f\in\CPMM$ whose transition matrix does not
admit any nonnegative eigenvectors, summable or otherwise. That means
that $f$ is not conjugate to any map of any constant slope, even
allowing for maps on the extended real line. In terms of the
Vere-Jones recurrence hierarchy, our map $f$ has transient Markov
dynamics. Indeed it must, since for recurrent Markov chains there
always exists a nonnegative eigenvector (see~\cite{VJ}). For transient
Markov chains, Pruitt~\cite{Pr} offers a nice criterion for the
existence or nonexistence of a nonnegative eigenvector; our example
was inspired by Pruitt's paper. Other examples of this type are considered in
the forthcoming article~\cite{BB}.

We construct $f$ as follows. Fix a subset $N\subset\N$, $1\notin N$,
such that $\frac{\pi(n)}{n}\to0$, and
$\sum_{n=0}^{\infty}3^{-\pi(n)}<3$, where
$\pi(n)=\#N\cap\{1,2,\ldots,n\}$, $\pi(0)=0$. If we wish to be
explicit, we may take $N$ such that $\{\pi(n)\}_{n=0}^\infty$ is the
sequence 0, 0, 1, 2, 2, 3, 3, 3, 4, 4, 4, 4, \ldots. Subdivide $[0,1]$
into adjacent intervals
$I_n=\left[\frac{1}{2^{n+1}},\frac{1}{2^n}\right]$, $n\geq0$. Let
$f:[0,1]\to[0,1]$ be the continuous, piecewise affine map with the
following properties. For $n\in N$, $f$ maps $I_n$ onto $I_{n-1}$ once
with slope $+2$. For $n\in\N\setminus N$ $f$ maps $I_n$ onto $I_{n-1}$
three times with alternating slopes $\pm6$. Finally, $f$ maps $I_0$
onto the whole space $[0,1]$ with slope $-2$. The idea is illustrated
in Figure~\ref{fig:noconstslope}; the choice of $N$ controls which
windows contain only one branch of monotonicity.

\begin{figure}[ht]
\begin{center}
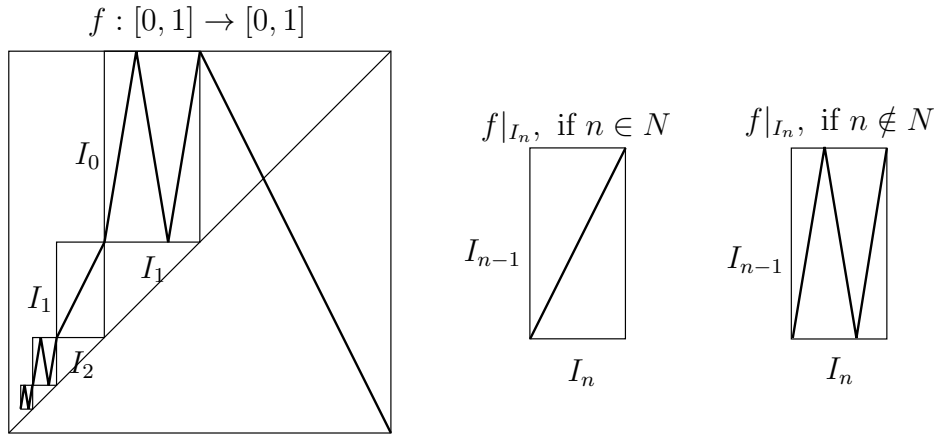
\caption{A map not conjugate to any constant slope
  map.}
\end{center}\label{fig:noconstslope}
\end{figure}

If we further subdivide each of the intervals $I_n$, $n\geq1$, into
three subintervals
$I_n^k=[\frac{1}{2^{n+1}}+\frac{k}{3\cdot 2^{n+1}},
\frac{1}{2^{n+1}}+\frac{k+1}{3\cdot 2^{n+1}}]$,
$k\in\{0,1,2\}$, then our map $f$ is countably piecewise monotone and
Markov with respect to this refined partition.

We wish to use Theorem~\ref{thm:main}, and therefore we must verify 
that $f$ is topologically mixing.  Let $U$ be an arbitrary open interval.
If an interval is mapped forward monotonically by $f$, then its image is an 
interval of at least twice the length.  Therefore, there is some minimal $n_0$ 
such that $f^{n_0}(U)$ contains a folding point of $f$.  Then $f^{n_0+1}(U)$
contains a point of the form $2^{-n_1}$.  Thus, $f^{n_0+n_1+2}(U)$ contains a
neighborhood of zero, and hence a whole interval $I_{n_2}$.  Then
$f^{n_0+n_1+n_2+3}(U)=[0,1]$.  This shows that $f$ is topologically
mixing (and even locally eventually onto).

We investigate now the
existence of nonnegative eigenvectors for the corresponding transition
matrix. Suppose that there is an eigenvector with some eigenvalue
$\lambda>0$. Let $v_n$ denote the sum of the entries corresponding to
$I_n^1$, $I_n^2$, and $I_n^3$, and let $v_0$ denote the entry
corresponding to the undivided interval $I_0$. Then the eigenvector
condition implies that
\begin{equation}\label{noeigenvector}
\lambda v_n = \begin{cases} v_{n-1}, & \textnormal{if }n\in N
  \\ 3v_{n-1}, & \textnormal{if }n\notin N \end{cases}, \qquad \lambda
v_0 = \sum_{n=0}^\infty v_n.
\end{equation}
By rescaling our vector if necessary, we may suppose that $v_0=1$. It
follows inductively that $v_n=\frac{3^{n-\pi(n)}}{\lambda^n}$. If
$\lambda\geq3$, then $\sum_{n=0}^\infty v_n\leq\sum_{n=0}^\infty 3^{-\pi(n)} < 3$ by the
choice of $N$, which contradicts the last equation
of~\eqref{noeigenvector}. If $\lambda<3$, then $v_n \to \infty$ by the
choice of $N$, so that $\sum_n v_n$ diverges, which again
contradicts~\eqref{noeigenvector}. It follows that our transition
matrix has no nonnegative eigenvectors. In light of
Theorem~\ref{thm:main}, this means that there does not exist any
conjugate map of any constant slope, even allowing for maps on the
extended real line.

\section{Piecewise Continuous Case}\label{sec:pcws}

We turn our attention now to piecewise continuous maps. We finish the
proof of Theorem~\ref{thm:pcws}, showing by example the insufficiency
of condition~\eqref{criterion}. We begin by defining two maps $f$, $g$
on the extended half-line $[0,\infty]$ with the same Markov structure.
The $P$-basic intervals are the intervals
$A_i=\left(2^i+2i-3,2^{i+1}+2i-2\right)$,
$B_i=\left(2^{i+1}+2i-2,2^{i+1}+2i-1\right)$, $i=0,1,2,\ldots$; thus
$P$ is the set of endpoints of these intervals together with the point
at infinity. Notice that these intervals have lengths $|A_i|=2^i+1$,
$|B_i|=1$ and are arranged from left to right in the order $A_0 < B_0
< A_1 < B_1 < A_2 < B_2 < \ldots$. Both maps $f$, $g$ will exhibit
Markov transitions as indicated in Figure~\ref{fig:transition}, where
an arrow $I\to J$ indicates that the image of interval $I$ includes
interval $J$.
\begin{figure}[htb!]
\begin{center}
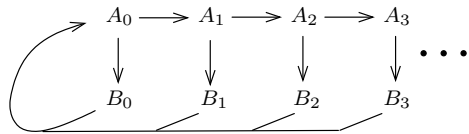
\end{center}
\caption{Transition diagram for the maps $f$ and $g$.}
\label{fig:transition}
\end{figure}
For each $i$, both $g|_{A_i}$ and $g|_{B_i}$ are affine with slope
$2$; this completes the definition of $g$. Moreover, for all $i$,
$f|_{B_i}$ is affine with slope $2$. However, the definition of
$f|_{A_i}$ is different. For each $i$, $f$ carries
$\left(2^i+2i-3,2^{i+1}+2i-3\right)$ (all but the right-most unit of
$A_i$) affinely onto $B_i$ with slope $2^{-i}$, and carries
$\left(2^{i+1}+2i-3,2^{i+1}+2i-2\right)$ (the right-most unit of
$A_i$) affinely onto $A_{i+1}$ with slope $2^{i+1}+1$. This completes
the definition of $f$. The graphs of both maps are shown below.

\begin{figure}[htb!]
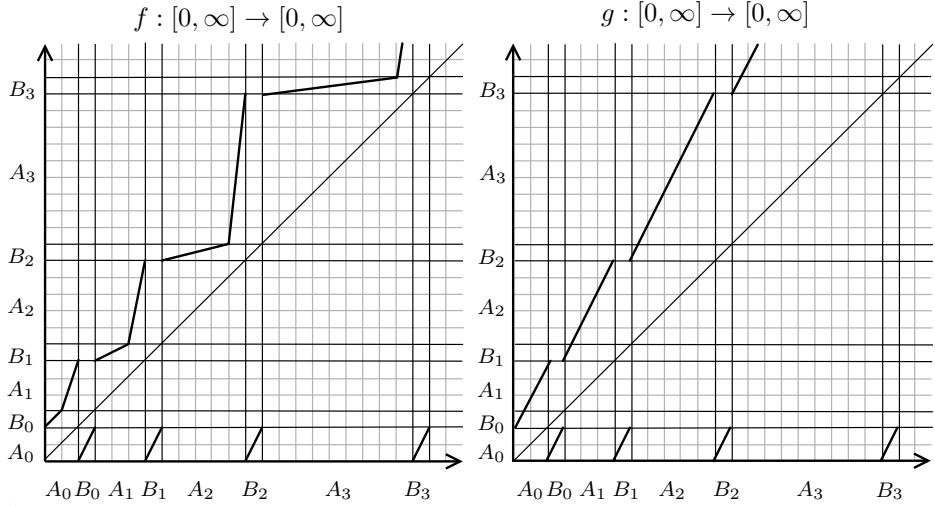
\caption{The graphs of the maps $f$ and $g$.}\label{fig:graphs}
\end{figure}

We claim that the map $f$ is topologically mixing. 
Indeed, let $U, V\subset[0,\infty]$ be any
open intervals. If for each $i$, $f^i(U)$ is contained in a single
$P$-basic interval, then either there are infinitely many indices
$i_j$ for which $f^{i_j}(U)\subset A_0$ or there are only finitely
many (perhaps zero). The first case is impossible because the length
$|f^{i_{j+1}}(U)|$ must be greater than the length $|f^{i_j}(U)|$ by a
factor of at least $2$ (consider all possible loops from $A_0$ to
itself in Figure~\ref{fig:transition} and multiply slopes along the
loop). The second case is impossible because then some image of $U$
must be contained in a set of the form $\bigcap_{k=0}^\infty f^{-k}
(A_{n+k})$ for some $n$ (look again at Figure~\ref{fig:transition}). But
the length of the finite intersection $\bigcap_{k=0}^K f^{-k}(A_{n+k})$
is equal to the length of $A_n$ times the proportion of $A_n$ mapped
onto $A_{n+1}$, times the proportion of $A_{n+1}$ mapped onto $A_{n+2}$,
and so on, up to the proportion of $A_{n+K-1}$ mapped onto $A_{n+K}$,
that is,
\begin{multline}\label{lengthcalculation}
\left|\bigcap_{k=0}^K f^{-k} A_{n+k}\right|=
|A_n| \times \frac{1}{|A_n|} \times \frac{1}{|A_{n+1}|} \times \cdots \times
 \frac{1}{|A_{n+K-1}|}= \\
= \frac{2^n+1}{\left(2^n+1\right)\left(2^{n+1}+1\right)\cdots
\left(2^{n+K-1}+1\right)}
\end{multline}
which decreases to zero as $K\to\infty$. We conclude that there
exists $i$ such that $f^i(U)$ is not contained in a single $P$-basic
interval. Looking now at the graph of $f$, we can see that either
$f^{i+1}(U)$ or else $f^{i+2}(U)$ contains some interval
$(0,\eps)$, and in particular, contains some interval
$(0,2^{-i'})$. Observe that $f^{2i'}(0,2^{-i'})=A_0$. Let $i''$ be
such that $V\cap (A_{i''}\cup B_{i''})\neq\emptyset$. For all $n\geq
i''+3$ we have $f^n(A_0)\supset (A_{i''}\cup B_{i''})$ (look again for
paths in Figure~\ref{fig:transition}). Therefore, for all $n\geq
i+2+2i'+i''+3$ we have $f^n(U)\cap V \neq \emptyset$. This
 concludes the proof that $f$ is topologically mixing.

Next we consider eigenvectors associated with the Markov structure of
the map $f$. We single out the eigenvalue $\lambda=2$, and we denote
by $a_i, b_i$, $i=0,1,2,\ldots$, the entries of an eigenvector
corresponding to the intervals $A_i, B_i$, $i=0,1,2,\ldots$,
respectively. In particular, we must find nonnegative solutions to the
infinite system of equations
\begin{equation*}
\left\{
\begin{array}{l}
2a_i=b_i+a_{i+1} \\
2b_i=a_0
\end{array}
\right. ,\qquad i\in\mathbb{N}.
\end{equation*}
Since eigenvectors are defined only up to a scaling constant, we are
free to fix $b_0=1$. It follows that $a_0=2$ and that $b_i=1$ for all
$i$. Then the entries $a_i$ can be computed inductively as
$a_i=2^i+1$. Up to scaling, this is the only eigenvector for the
eigenvalue $\lambda=2$.

Despite the existence of this eigenvector, $f$ is not conjugate to any
map of constant slope $2$. Indeed, let $\varphi$ be a homeomorphism of
$[0,\infty]$ onto a closed (sub)interval of the extended real line;
without loss of generality we may assume that $\varphi$ is
orientation-preserving. Suppose that the conjugate map $\varphi\circ
f\circ\varphi^{-1}$ has constant slope $2$. By Lemma~\ref{lem:Bobok} the lengths of the
$\varphi(P)$-basic intervals must be given by an eigenvector with
eigenvalue $2$. Therefore, after rescaling and translating $\varphi$ if
necessary, the map $\varphi\circ f \circ\varphi^{-1}$ is equal to the map
$g$ which we have already defined, so that $\varphi(A_i)=A_i$ and
$\varphi(B_i)=B_i$ for all $i$. In other words, $\varphi$ fixes the entire
set $P$. Let $x_k$ denote the left-hand endpoint of the interval $A_0
\cap f^{-1}(A_1) \cap \ldots \cap f^{-k}(A_k)$, and let $y_k$ denote
the left-hand endpoint of the interval $A_0 \cap g^{-1}(A_1) \cap
\ldots \cap g^{-k}(A_k)$. Since $\varphi$ conjugates $f$ with $g$, we
must have $\varphi(x_k)=y_k$ for all $k$. By the same reasoning as
we used to derive equation~\eqref{lengthcalculation}, we have
\begin{equation*}
\begin{array}{l}
x_k = 2 - 2 \times \frac{1}{2} \times \frac{1}{3} \times \cdots \times
\frac{1}{2^{k-1}+1}, \\ 
y_k = 2 - 2 \times \frac{3}{4} \times \frac{5}{6}
\times \cdots \times \frac{2^k+1}{2^k+2}.
\end{array}
\end{equation*}
Inductively, $y_k=2 - \frac{2^k+1}{2^k}$.  We have $x_k\to2$ and 
$y_k\to1$. Since $\varphi(2)=2$, this
contradicts the continuity of $\varphi$. We conclude that $f$ is not
conjugate to any map of constant slope $2$.

\begin{remark}
We can also interpret this example from the point 
of view of wandering intervals.  The map $g$ has a wandering interval 
$A_0\cap g^{-1}(A_1)\cap g^{-2}(A_2) \cap \cdots$.  Thus, even though 
it has constant slope $2$ and the right Markov structure, it cannot be 
conjugate to the topologically mixing map $f$.
\end{remark}

\section{The Mixing Hypothesis}\label{sec:mix}

We turn our attention now to maps which are topologically transitive
but not topologically mixing. We finish the proof of
Theorem~\ref{thm:notmixing}, showing the insufficiency of
condition~\eqref{criterion}. We construct a map $\tilde{f}$ in $\CPMM$ which
is topologically transitive but not mixing. We give a nonnegative
eigenvector $v$ for the transition matrix $T$, but prove that $\tilde{f}$ is
not conjugate to any map on any subinterval
$[a,b]\subset[-\infty,\infty]$ with constant slope equal to the eigenvalue of
$v$.

Let $f$ and $P$ be as defined in Section~\ref{sec:extended}.  We define
$\tilde{f}:[-1,1]\to[-1,1]$ by the formula
\begin{equation*}
\tilde{f}(x)=\begin{cases}-f(x),&\textnormal{ if }x\in[0,1]\\-x&\textnormal{
  if }x\in[-1,0]\end{cases}
\end{equation*}
This map $\tilde{f}$ is piecewise monotone and Markov with respect to the set
$\tilde{P}=P\cup -P$. Figure~\ref{fig:SquareRoot} shows the graph of $\tilde{f}$
(in bold). Superimposed is the graph of
the second iterate $\tilde{f}^2$. By construction, $\tilde{f}^2|_{[0,1]}$ and
$\tilde{f}^2|_{[-1,0]}$ are both isomorphic copies of the map $f$.  In this sense,
$\tilde{f}$ is a kind of dynamical square root of $f$.

\begin{figure}[hbt!]
\centering
\includegraphics[scale=.3]{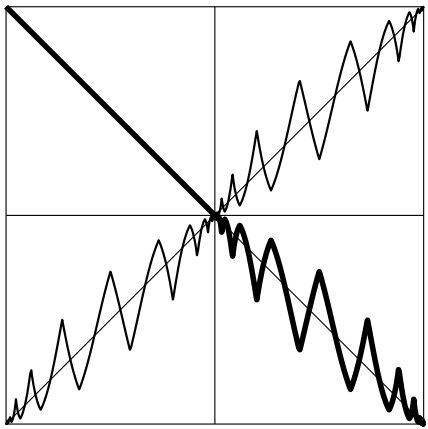}
\caption{The map $\tilde{f}$ (shown in bold) is transitive but not
  mixing.  It's second iterate (superimposed) consists of two
  copies of the map $f$ from Section~\ref{sec:extended}
  (compare with Figure~\ref{fig:F}).}
\label{fig:SquareRoot}
\end{figure}

We claim that $\tilde{f}$ is topologically transitive, but not topologically
mixing. To see the transitivity, take arbitrary nonempty
open subsets $U$, $V$ of $[-1,1]$. After shrinking these sets, we may assume
that $U,V$ are open intervals not containing $0$. Consider first the case when $U,V\subset[0,1]$. By
the transitivity of $f$ there exists $n$ such that $U\cap
f^{-n}(V)\neq\emptyset$, but then $U\cap \tilde{f}^{-2n}(V)\neq\emptyset$. The
case when $U,V\subset[-1,0]$ is similar. Now consider the case when
$U\subset[0,1]$ and $V\subset[-1,0]$. Using the reflected set $-V$ and
the transitivity of $f$, find $n$ such that $U\cap
f^{-n}(-V)\neq\emptyset$. Then $U\cap \tilde{f}^{2n-1}(V)\neq\emptyset$. The
case when $U\subset[-1,0]$ and $V\subset[0,1]$ is similar. This shows
topological transitivity of $\tilde{f}$. To see that $\tilde{f}$ is not topologically
mixing, notice that the set $\{n\in\N:(0,1)\cap
\tilde{f}^{-n}(0,1)\neq\emptyset\}$ consists of only the even natural numbers.

Let $\tilde{T}$ be the 0-1 transition matrix for the map $\tilde{f}$ with respect
to the Markov partition by $\tilde{P}$.  We label the $\tilde{P}$-basic intervals
$I_k, J_k, k\in\Z$, where the intervals $I_k$ are given by equation~\eqref{enum} and
the intervals $J_k$ are their reflections, $J_k=-I_k$. Fix $\lambda=2+\sqrt{5}$.
We find all nonnegative solutions $v\in\R^{\B(\tilde{P})}$ to the equation
$\tilde{T}v=\sqrt\lambda v$.  In light of the Markov transitions, this is the infinite
system of equations
\begin{equation}\label{InfSys2}
\def\arraystretch{1.5}
\left\{
\begin{array}{lcl}
\sqrt\lambda\, v_{I_{2k}}&=&\sum_{i=2k-2}^{2k+2}v_{J_i} \\
\sqrt\lambda\, v_{I_{2k+1}}&=&\sum_{i=2k}^{2k+2}v_{J_i} \\
\sqrt\lambda\, v_{J_k}&=&v_{I_k}
\end{array}
\right.
\qquad \qquad k\in\Z
\end{equation}

If we substitute the
last line in equation~\eqref{InfSys2} into the first two lines, we recover
equation~\eqref{InfSys}, which for $\lambda=2+\sqrt{5}$ has (up to scalar
multiples) the unique nonnegative solution~\eqref{InfSysSol}.
Therefore~\eqref{InfSys2} has (up to scalar multiples) the unique 
nonnegative solution

\begin{equation}\label{InfSys2Sol}
v_{I_{2k}}=2, \quad
v_{I_{2k+1}}=\sqrt5-1, \quad
v_{J_{2k}}=\frac{2}{\sqrt\lambda}, \quad
v_{J_{2k+1}}=\frac{\sqrt5-1}{\sqrt\lambda}, \qquad
k\in\Z.
\end{equation}

Now we show that despite the existence of this eigenvector $v$, there
does not exist any conjugacy $\psi$ of the map $\tilde{f}$ to a map of
constant slope $\sqrt\lambda$. Assume the contrary. Then by the
uniqueness of $v$ and by Lemma~\ref{lem:Bobok}, we have
\begin{multline*}
|\psi(I_{2k})|=2c, \quad
|\psi(I_{2k+1})| = (\sqrt5-1)c, \\
|\psi(J_{2k})|=\frac{2c}{\sqrt\lambda}, \quad
|\psi(J_{2k+1})|=\frac{(\sqrt5-1)c}{\sqrt\lambda}, \qquad k\in\Z,
\end{multline*}
for some positive real scalar $c$. But the $P$-basic intervals
accumulate at the center of $[-1,1]$ so that a small open interval
$(-\eps,\eps)$ contains infinitely many $P$-basic intervals. Thus,
$\psi(-\eps,\eps)$ has infinite length. On the other hand, a
nondecreasing homeomorphism $\psi:[-1,1]\to[-\infty,\infty]$ must take
finite values at every interior point of the interval $[-1,1]$. This
is a contradiction.


\end{document}